\documentclass{amsart}
\usepackage[dvips,final]{graphics}
\usepackage{array}
\usepackage{arydshln}
\usepackage[makeroom]{cancel}
 \usepackage[all]{xy}
 \usepackage{url}
\usepackage{multirow, blkarray}
\usepackage{booktabs}
\usepackage{textcomp}
 \usepackage[final]{epsfig}
 \usepackage{color}
\usepackage[T1]{fontenc}      
\usepackage[english,french]{babel}
\usepackage[utf8]{inputenc}
\usepackage{blindtext}

\usepackage{amsfonts,amscd,array, mathdots, epigraph}
\usepackage{amsmath}
\usepackage{amssymb}
\usepackage{amsthm}
\usepackage{mathrsfs}
\usepackage{stmaryrd}
\usepackage{slashbox}
\usepackage{diagbox}
\usepackage{enumitem}

\begin{document}


\newtheorem{theorem}{Théorème}[section]
\newtheorem{theore}{Théorème}
\newtheorem{definition}[theorem]{Définition}
\newtheorem{proposition}[theorem]{Proposition}
\newtheorem{corollary}[theorem]{Corollaire}
\newtheorem{con}{Conjecture}
\newtheorem*{remark}{Remarque}
\newtheorem*{remarks}{Remarques}
\newtheorem*{pro}{Problème}
\newtheorem*{examples}{Exemples}
\newtheorem*{example}{Exemple}
\newtheorem{lemma}[theorem]{Lemme}


\title{Combinatoire des sous-groupes de congruence du groupe modulaire II}

\author{Flavien Mabilat}

\date{}


\address{
Flavien Mabilat,
Laboratoire de Mathématiques de Reims,
UMR9008 CNRS et Université de Reims Champagne-Ardenne, 
U.F.R. Sciences Exactes et Naturelles 
Moulin de la Housse - BP 1039 
51687 Reims cedex 2,
France
}
\email{flavien.mabilat@univ-reims.fr}

\maketitle

\selectlanguage{french}
\begin{abstract}
Dans cet article, on souhaite étudier la combinatoire des sous-groupes de congruence du groupe modulaire. Pour cela, on considère une équation matricielle liée à celle qui apparaît lors de l'étude des frises de Coxeter et on étudie ces solutions irréductibles. En particulier, on donne de nouvelles propriétés des solutions monomiales minimales. De plus, on introduit la notion de solutions dynomiales minimales et on donne des conditions suffisantes d'irréductibilité pour celles-ci. 
\\
\end{abstract}

\selectlanguage{english}
\begin{abstract}
In this paper, we study combinatorics of congruence subgroups of the modular group. More precisely, we consider the matrix equation that naturally arises in the theory of Coxeter friezes and investigate its irreducible solutions. We give new properties for minimal monomial solutions. Furthermore, we introduce the notion of minimal dynomial solutions and study their irreducibility.

\end{abstract}

\selectlanguage{french}

\thispagestyle{empty}

{\bf Mots clés:} groupe modulaire; sous-groupes de congruence; irréductibilité
\\
\begin{flushright}
\og \textit{Il est très difficile d'imaginer quelque chose de simple.} \fg
\\ Pierre MacOrlan , \textit{Villes}
\end{flushright}

\section{Introduction}

L'une des propriétés les plus intéressantes du groupe modulaire 
\[SL_{2}(\mathbb{Z})=
\left\{
\begin{pmatrix}
a & b \\
c & d
   \end{pmatrix}
 \;\vert\;a,b,c,d \in \mathbb{Z},\;
 ad-bc=1
\right\}\] est l'existence de parties génératrices à deux éléments. On peut, en particulier, prendre les deux matrices suivantes (voir par exemple \cite{A}): \[T=\begin{pmatrix}
 1 & 1 \\[2pt]
    0    & 1 
   \end{pmatrix}, S=\begin{pmatrix}
   0 & -1 \\[2pt]
    1    & 0 
   \end{pmatrix}.
 \] 
\\À partir de ce choix, on peut montrer que pour toute matrice $A$ de $SL_{2}(\mathbb{Z})$ il existe un entier strictement positif $n$ et des entiers strictement positifs $a_{1},\ldots,a_{n}$ tels que \[A=T^{a_{n}}ST^{a_{n-1}}S\cdots T^{a_{1}}S=\begin{pmatrix}
   a_{n} & -1 \\[4pt]
    1    & 0 
   \end{pmatrix}
\begin{pmatrix}
   a_{n-1} & -1 \\[4pt]
    1    & 0 
   \end{pmatrix}
   \cdots
   \begin{pmatrix}
   a_{1} & -1 \\[4pt]
    1    & 0 
    \end{pmatrix}.\]
On utilisera la notation $M_{n}(a_{1},\ldots,a_{n})$ pour désigner la matrice \[\begin{pmatrix}
   a_{n} & -1 \\[4pt]
    1    & 0 
   \end{pmatrix}
\begin{pmatrix}
   a_{n-1} & -1 \\[4pt]
    1    & 0 
   \end{pmatrix}
   \cdots
   \begin{pmatrix}
   a_{1} & -1 \\[4pt]
    1    & 0 
    \end{pmatrix}.\] Ces matrices interviennent également dans la théorie des frises de Coxeter (voir par exemple \cite{Cox,CH} pour la définition des frises de Coxeter). En effet, les solutions de l'équation $M_{n}(a_{1},\ldots,a_{n})=-Id$ permettent de construire des frises de Coxeter, et, à partir d'une telle frise, on peut obtenir une solution de cette équation (voir \cite{BR} et \cite{CH} proposition 2.4). Les frises de Coxeter possèdent par ailleurs des connections avec de nombreux autres domaines mathématiques (voir par exemple \cite{Mo1}).
\\
\\Ceci amène naturellement à l'étude de l'équation généralisée suivante: \begin{equation}
\label{a}
M_{n}(a_1,\ldots,a_n)=\pm Id.
\end{equation} V.Ovsienko (voir \cite{O} Théorèmes 1 et 2) a résolu celle-ci sur $\mathbb{N}^{*}$ et donné une description combinatoire des solutions en terme de découpages de polygones (généralisant par ailleurs un théorème antérieur dû à Conway et Coxeter, voir \cite{CoCo,Cox,Hen}). On dispose également des solutions de cette équation sur $\mathbb{N}$ (voir \cite{C} Théorème 3.1), sur $\mathbb{Z}$ (voir \cite{C} Théorème 3.2) et sur $\mathbb{Z}[\alpha]$ avec $\alpha$ un nombre complexe transcendant (voir \cite{Ma2} Théorème 2.7). On peut aussi étudier l'équation \eqref{a} en remplaçant $\pm Id$ par $\pm M$ avec $M$ une matrice du groupe modulaire, notamment pour $M=S$ et $M=T$ (voir \cite{Ma3}).
\\
\\On va s'intéresser ici aux cas des anneaux $\mathbb{Z}/N\mathbb{Z}$, c'est-à-dire à la résolution sur $\mathbb{Z}/N\mathbb{Z}$ de l'équation :
\begin{equation}
\label{p}
\tag{$E_{N}$}
M_{n}(a_1,\ldots,a_n)=\pm Id.
\end{equation} On dira, en particulier, qu'une solution de \eqref{p} est de taille $n$ si cette solution est un $n$-uplet d'éléments de $\mathbb{Z}/N\mathbb{Z}$. L'étude de l'équation ci-dessus permet notamment de chercher toutes les écritures des éléments des sous-groupes de congruence ci-dessous: \[\hat{\Gamma}(N)=\{A \in SL_{2}(\mathbb{Z})~{\rm tel~que}~A= \pm Id~( {\rm mod}~N)\}\] sous la forme $M_{n}(a_1,\ldots,a_n)$ avec les $a_{i}$ des entiers strictement positifs. 
\\
\\L'équation \eqref{p} a déjà été étudiée dans des travaux précédents (voir \cite{Ma1,M}). Pour mener à bien cette étude on avait utilisé une notion de solutions irréductibles à partir de laquelle on peut construire l'ensemble des solutions (voir section suivante). Ceci nous avait permis de résoudre complètement \eqref{p} pour $N \leq 6$ (voir \cite{M} section 4). On avait également obtenu des résultats généraux d'irréductibilité en définissant notamment la notion de solutions monomiales minimales (voir \cite{M} section 3.3 et la section suivante). D'autres éléments pouvant être reliés aux cas des anneaux $\mathbb{Z}/N\mathbb{Z}$, avec $N$ premier, peuvent également être trouvés dans \cite{Mo2}.
\\
\\L'objectif ici est de poursuivre cette étude en obtenant des résultats sur la taille et l'irréductibilité des solutions monomiales minimales (voir section \ref{MM}) et d'obtenir des résultats d'irréductibilité pour d'autres solutions (voir section \ref{DM}).

\section{Définitions et résultats principaux}\label{RP}    

L'objectif de cette section est de rappeler les définitions, introduites notamment dans \cite{C} et \cite{M}, utiles à l'étude de l'équation \eqref{p} et d'énoncer les résultats principaux de ce texte. Sauf mention contraire, $N$ désigne un entier naturel supérieur à $2$ et si $a \in \mathbb{Z}$ on note $\overline{a}=a+N\mathbb{Z}$. 

\begin{definition}[\cite{WZ}, définition 1.8]
\label{21}

Soient $(\overline{a_{1}},\ldots,\overline{a_{n}}) \in (\mathbb{Z}/N \mathbb{Z})^{n}$ et $(\overline{b_{1}},\ldots,\overline{b_{m}}) \in (\mathbb{Z}/N \mathbb{Z})^{m}$. On définit l'opération ci-dessous: \[(\overline{a_{1}},\ldots,\overline{a_{n}}) \oplus (\overline{b_{1}},\ldots,\overline{b_{m}})= (\overline{a_{1}+b_{m}},\overline{a_{2}},\ldots,\overline{a_{n-1}},\overline{a_{n}+b_{1}},\overline{b_{2}},\ldots,\overline{b_{m-1}}).\] Le $(n+m-2)$-uplet obtenu est appelé la somme de $(\overline{a_{1}},\ldots,\overline{a_{n}})$ avec $(\overline{b_{1}},\ldots,\overline{b_{m}})$.

\end{definition}

\begin{examples}

{\rm Voici quelques exemples de sommes :
\begin{itemize}
\item $(\overline{1},\overline{2},\overline{3}) \oplus (\overline{4},\overline{1},\overline{3},\overline{2})= (\overline{3},\overline{2},\overline{7},\overline{1},\overline{3})$;
\item $(\overline{4},\overline{0},\overline{1},\overline{2}) \oplus (\overline{-1},\overline{0},\overline{1}) = (\overline{5},\overline{0},\overline{1},\overline{1},\overline{0})$;
\item $n \geq 2$, $(\overline{a_{1}},\ldots,\overline{a_{n}}) \oplus (\overline{0},\overline{0}) = (\overline{0},\overline{0}) \oplus (\overline{a_{1}},\ldots,\overline{a_{n}})=(\overline{a_{1}},\ldots,\overline{a_{n}})$.
\end{itemize}
}
\end{examples}

En particulier, si $(\overline{b_{1}},\ldots,\overline{b_{m}})$ est une solution de \eqref{p} alors la somme $(\overline{a_{1}},\ldots,\overline{a_{n}}) \oplus (\overline{b_{1}},\ldots,\overline{b_{m}})$ est solution de \eqref{p} si et seulement si $(\overline{a_{1}},\ldots,\overline{a_{n}})$ est solution de \eqref{p} (voir \cite{C,WZ} et \cite{M} proposition 3.7). En revanche, l'opération $\oplus$ n'est ni commutative ni associative (voir \cite{WZ} exemple 2.1). 

\begin{definition}[\cite{WZ}, définition 1.5]
\label{22}

 Soient $(\overline{a_{1}},\ldots,\overline{a_{n}}) \in (\mathbb{Z}/N \mathbb{Z})^{n}$ et $(\overline{b_{1}},\ldots,\overline{b_{n}}) \in (\mathbb{Z}/N \mathbb{Z})^{n}$. On dit que $(\overline{a_{1}},\ldots,\overline{a_{n}}) \sim (\overline{b_{1}},\ldots,\overline{b_{n}})$ si $(\overline{b_{1}},\ldots,\overline{b_{n}})$ est obtenu par permutation circulaire de $(\overline{a_{1}},\ldots,\overline{a_{n}})$ ou de $(\overline{a_{n}},\ldots,\overline{a_{1}})$.

\end{definition}

On montre facilement que $\sim$ est une relation d'équivalence sur les $n$-uplets d'éléments de $\mathbb{Z}/N \mathbb{Z}$ (voir \cite{WZ}, lemme 1.7). De plus, si $(\overline{a_{1}},\ldots,\overline{a_{n}}) \sim (\overline{b_{1}},\ldots,\overline{b_{n}})$ alors $(\overline{a_{1}},\ldots,\overline{a_{n}})$ est solution de \eqref{p} si et seulement si $(\overline{b_{1}},\ldots,\overline{b_{n}})$ est solution de \eqref{p} (voir \cite{C} proposition 2.6).

\begin{definition}[\cite{C}, définition 2.9]
\label{222}

Une solution $(\overline{c_{1}},\ldots,\overline{c_{n}})$ avec $n \geq 3$ de \eqref{p} est dite réductible s'il existe une solution de \eqref{p} $(\overline{b_{1}},\ldots,\overline{b_{l}})$ et un $m$-uplet $(\overline{a_{1}},\ldots,\overline{a_{m}})$ d'éléments de $\mathbb{Z}/N \mathbb{Z}$ tels que \begin{itemize}
\item $(\overline{c_{1}},\ldots,\overline{c_{n}}) \sim (\overline{a_{1}},\ldots,\overline{a_{m}}) \oplus (\overline{b_{1}},\ldots,\overline{b_{l}})$;
\item $m \geq 3$ et $l \geq 3$.
\end{itemize}
Une solution est dite irréductible si elle n'est pas réductible. 

\end{definition}

\begin{remark} 

{\rm On ne considère pas $(\overline{0},\overline{0})$ comme une solution irréductible de \eqref{p}.}

\end{remark}

\indent L'un de nos objectifs principaux est de trouver des solutions irréductibles de l'équation \eqref{p}. En particulier, on a introduit dans \cite{M} la notion de solutions monomiales rappelée ci-dessous :

\begin{definition}
\label{23}

i)~Soient $n \in \mathbb{N}^{*}$ et $\overline{k} \in \mathbb{Z}/N\mathbb{Z}$. On appelle solution $(n,\overline{k})$-monomiale un $n$-uplet d'éléments de $\mathbb{Z}/ N \mathbb{Z}$ constitué uniquement de $\overline{k} \in \mathbb{Z}/N\mathbb{Z}$ et solution de \eqref{p}.
\\
\\ ii)~On appelle solution monomiale une solution pour laquelle il existe $m \in \mathbb{N}^{*}$ et $\overline{l} \in \mathbb{Z}/N\mathbb{Z}$ tels qu'elle est $(m,\overline{l})$-monomiale.
\\
\\ iii)~On appelle solution $\overline{k}$-monomiale minimale une solution $(n,\overline{k})$-monomiale avec $n$ le plus petit entier pour lequel il existe une solution $(n,\overline{k})$-monomiale.
\\
\\ iv)~On appelle solution monomiale minimale une solution $\overline{k}$-monomiale minimale pour un $\overline{k} \in \mathbb{Z}/N\mathbb{Z}$.

\end{definition}

\noindent On connaît déjà un certain nombre de propriétés d'irréductibilité pour ces solutions. Celles-ci seront évoquées dans la section suivante où on démontrera également le résultat ci-dessous :

\begin{theorem}
\label{24}

Soit $N$ un entier pair supérieur à 4. On a deux cas :
\begin{itemize}
\item Si $4 \mid N$ alors la solution $\overline{\frac{N}{2}}$-monomiale minimale de \eqref{p} est de taille 4;
\item Si $4 \nmid N$ alors la solution $\overline{\frac{N}{2}}$-monomiale minimale de \eqref{p} est de taille 6.
\\
\end{itemize}
De plus, celle-ci est irréductible dans les deux cas.

\end{theorem}

\noindent Pour obtenir de nouveaux résultats d'irréductibilité on définit une nouvelle classe de solutions de \eqref{p}.

\begin{definition}
\label{25}

i)~Soient $n \in \mathbb{N}^{*}$ pair et $\overline{k} \in \mathbb{Z}/N\mathbb{Z}$. On appelle solution $(n,\overline{k})$-dynomiale une solution de taille $n$ de \eqref{p} de la forme $(\overline{k},\overline{-k},\ldots,\overline{k},\overline{-k})$.
\\
\\ ii)~On appelle solution dynomiale une solution pour laquelle il existe $m \in \mathbb{N}^{*}$ et $\overline{l} \in \mathbb{Z}/N\mathbb{Z}$ tels qu'elle est $(m,\overline{l})$-dynomiale.
\\
\\ iii)~On appelle solution $\overline{k}$-dynomiale minimale une solution $(n,\overline{k})$-dynomiale avec $n$ le plus petit entier pour lequel il existe une solution $(n,\overline{k})$-dynomiale.
\\
\\ iv)~On appelle solution dynomiale minimale une solution $\overline{k}$-dynomiale minimale pour un $\overline{k} \in \mathbb{Z}/N\mathbb{Z}$.
	
\end{definition}	

\begin{remark}  

{\rm Les propriétés suivantes des solutions dynomiales sont immédiates:
\begin{itemize}
\item Une solution dynomiale est toujours de taille paire.
\item Une solution $(n,\overline{k})$-dynomiale est équivalente à une solution $(n,\overline{-k})$-dynomiale.
\item Si $\overline{2k}=\overline{0}$ alors une solution $(n,\overline{k})$-dynomiale est une solution $(n,\overline{k})$-monomiale.
\end{itemize}}

\end{remark}

\noindent On dispose pour cette classe de solutions du résultat d'irréductibilité suivant démontré dans la section \ref{DM} :

\begin{theorem}
\label{26}

Soient $N$ un nombre premier supérieur à 5 et $\overline{k} \in \mathbb{Z}/N\mathbb{Z}$. On suppose que les deux conditions suivantes sont vérifiées : 
\begin{itemize}
\item  $\overline{k} \neq \overline{0}$;
\item $\overline{k}^{2}+\overline{8}$ n'est pas un carré dans $\mathbb{Z}/N\mathbb{Z}$.
\end{itemize}
La solution $\overline{k}$-dynomiale minimale de \eqref{p} est irréductible.

\end{theorem}

Ce théorème permet d'obtenir plusieurs résultats d'irréductibilité intéressants exposés dans la section \ref{app}. On montre notamment dans celle-ci que la solution $\overline{2}$-dynomiale minimale de \eqref{p} est irréductible lorsque $N$ est un nombre premier supérieur à 5 vérifiant $N \not\equiv \pm 1 [12]$.

\section{Propriétés des solutions monomiales minimales}
\label{MM}

Comme nous venons de l'évoquer, les solutions monomiales minimales possèdent un certain nombre de propriétés intéressantes (voir \cite{M} section 3.3). On dispose notamment des deux résultats d'irréductibilité énoncés ci-dessous :

\begin{theorem}
\label{31}

Soit $N$ un entier naturel supérieur à 2.
\\i)~(\cite{M}, Théorème 3.16) Si $N$ est premier alors les solutions monomiales minimales de \eqref{p} différentes de $(\overline{0},\overline{0})$ sont irréductibles.
\\ii)~(\cite{M}, Théorème 2.6) Si $N \geq 3$, $(\overline{2},\ldots,\overline{2}) \in (\mathbb{Z}/N\mathbb{Z})^{N}$ est une solution monomiale minimale irréductible de \eqref{p}.

\end{theorem}

L'objectif de cette section est d'approfondir l'étude de ces solutions en obtenant notamment des éléments sur leur taille et de nouveaux résultats d'irréductibilité. Dans cette partie, $N$ est un entier naturel supérieur ou égal à 2. 

\subsection{Taille des solutions monomiales minimales}

L'un des problèmes soulevés lors de l'étude de ces solutions (voir \cite{M} problème 1) était d'avoir des informations sur la taille des solutions monomiales minimales. Notre objectif ici est de fournir des éléments de réponse à ce problème.
\\
\\Par le théorème \ref{31}, on sait que la solution $\overline{2}$-monomiale minimale de \eqref{p} est de taille $N$. L'étude des solutions de \eqref{p} pour les petites valeurs de $n$ permet également de répondre précisément à notre question pour certaines valeurs de $\overline{k}$.

\begin{proposition}[\cite{M}, section 3.1]
\label{32}

i)~\eqref{p} n'a pas de solution de taille 1.
\\ii)~$(\overline{0},\overline{0})$ est l'unique solution de \eqref{p} de taille 2.
\\iii)~$(\overline{1},\overline{1},\overline{1})$ et $(\overline{-1},\overline{-1},\overline{-1})$ sont les seules solutions de \eqref{p} de taille 3.
\\iv)~Les solutions de \eqref{p} de taille 4 sont de la forme $(\overline{-a},\overline{b},\overline{a},\overline{-b})$ avec $\overline{ab}=\overline{0}$ et $(\overline{a},\overline{b},\overline{a},\overline{b})$ avec $\overline{ab}=\overline{2}$. 

\end{proposition}

On en déduit que la solution $\overline{0}$-monomiale minimale est de taille 2 et que les solutions $\pm\overline{1}$-monomiales minimales sont de taille 3. 
\\
\\De plus, pour tout $\overline{k} \in \mathbb{Z}/N\mathbb{Z}$, la solution $\overline{k}$-monomiale minimale et la solution $\overline{-k}$-monomiale minimale ont la même taille. Un simple calcul permet de montrer que la solution $\overline{3}$-monomiale minimale de $(E_{6})$ est de taille 6 et que la solution $\overline{3}$-monomiale minimale de $(E_{7})$ est de taille 4. Tout ceci nous permet de connaître précisément les tailles des solutions monomiales minimales pour $N \in \{2, 3, 4, 5, 6, 7\}$. 
\\
\\Pour le cas général on dispose du théorème ci-dessous qui donne une majoration de la taille :

\begin{theorem}[\cite{CG}, page 216]
\label{33}

Soit $N$ un entier naturel supérieur à 2. L'ordre des éléments de $SL_{2}(\mathbb{Z}/N\mathbb{Z})$ est inférieur à $3N$.

\end{theorem}

Dans le cas où $N$ est premier, on peut avoir des informations plus précises sur la taille des solutions monomiales minimales. La preuve qui suit est une adaptation à notre situation de la preuve du cas $N$ premier du théorème précédent fournie dans \cite{CG}.

\begin{theorem}
\label{34}

Soient $N$ un nombre premier impair et $\overline{k} \in \mathbb{Z}/N\mathbb{Z}$. On a deux cas :
\begin{itemize}
\item si $\overline{k}=\pm \overline{2}$ alors la solution $\overline{k}$-monomiale minimale est de taille $N$;
\item si $\overline{k} \neq \pm \overline{2}$ alors la taille de la solution $\overline{k}$-monomiale minimale divise $\frac{N-1}{2}$ ou $\frac{N+1}{2}$.
\end{itemize}

\end{theorem}

\begin{proof}

On montre, par récurrence, que pour tout $n$ dans $\mathbb{N}^{*}$ on a \[M_{n}(2,\ldots,2)=\begin{pmatrix}
   n+1   & -n \\
   n & -n+1
\end{pmatrix}.\] Donc, $M_{N}(\overline{2},\ldots,\overline{2})=\begin{pmatrix}
   \overline{N+1}   & \overline{-N} \\
   \overline{N} & \overline{-N+1}
\end{pmatrix}=Id$ ce qui implique que la taille de la solution $\overline{2}$-monomiale minimale de \eqref{p} divise $N$. De plus, \eqref{p} n'a pas de solution de taille 1 et $N$ est premier. Donc, la taille de la solution $\overline{2}$-monomiale minimale de \eqref{p} est égale à $N$. Celle de la solution $\overline{-2}$-monomiale minimale de \eqref{p} est par conséquent aussi égale à $N$. 
\\
\\On suppose maintenant $\overline{k} \neq \pm \overline{2}$. Le polynôme caractéristique de $M_{1}(\overline{k})$ est $\chi(X)=X^{2}-\overline{k}X+\overline{1}$. Ce polynôme a pour discriminant $\Delta=\overline{k}^{2}-\overline{4}=(\overline{k}-\overline{2})(\overline{k}+\overline{2}) \neq \overline{0}$ (car $\mathbb{Z}/N\mathbb{Z}$ est intègre). On a donc deux cas :
\begin{itemize}
\item $\Delta$ est un carré dans $\mathbb{Z}/N\mathbb{Z}$. Dans ce cas, $M_{1}(\overline{k})$ a deux valeurs propres distinctes et donc est diagonalisable dans $\mathbb{Z}/N\mathbb{Z}$. Notons $\overline{a}$ et $\overline{b}$ ses valeurs propres. Il existe $P \in GL_{2}(\mathbb{Z}/N\mathbb{Z})$ tel que \[M_{1}(\overline{k})=P\begin{pmatrix}
   \overline{a} & \overline{0} \\[2pt]
    \overline{0}    & \overline{b} 
   \end{pmatrix}P^{-1}.\] De plus, $\mathbb{Z}/N\mathbb{Z}-\{\overline{0}\}$ est un groupe de cardinal $N-1$. Ainsi, $\overline{a}^{N-1}=\overline{1}$ ($\overline{a} \neq \overline{0}$ puisque $\overline{ab}=\overline{1}$). Donc, $(\overline{a}^{\frac{N-1}{2}}-\overline{1})(\overline{a}^{\frac{N-1}{2}}+\overline{1})=\overline{0}$. On en déduit que $\overline{a}^{\frac{N-1}{2}}=\pm \overline{1}$ (puisque $\mathbb{Z}/N\mathbb{Z}$ est intègre). De même, $\overline{b}^{\frac{N-1}{2}}=\pm \overline{1}$. Or, $\overline{ab}=\overline{1}$ donc $\overline{a}^{\frac{N-1}{2}}=\overline{b}^{\frac{N-1}{2}}=\pm \overline{1}$. Il en découle \[M_{1}(\overline{k})^{\frac{N-1}{2}}=P\begin{pmatrix}
   \overline{a}^{\frac{N-1}{2}} & \overline{0} \\[2pt]
    \overline{0}    & \overline{b}^{\frac{N-1}{2}} 
   \end{pmatrix}P^{-1}=\pm Id.\]
Par conséquent, l'ordre de $M_{1}(\overline{k})$ dans $PSL_{2}(\mathbb{Z}/N\mathbb{Z})$, c'est-à-dire la taille de la solution $\overline{k}$-monomiale minimale de \eqref{p}, divise $\frac{N-1}{2}$.
\\
\item $\Delta$ n'est pas un carré dans $\mathbb{Z}/N\mathbb{Z}$. Soit $K$ un corps de décomposition de $\chi$ (voir \cite{G} Théorème V.18). $\chi$ a deux racines distinctes dans $K$. Notons les $x$ et $y$. On a 
\begin{eqnarray*}
(X-x^{N})(X-y^{N}) &=& X^{2}-(x^{N}+y^{N})X+x^{N}y^{N} \\
                   &=& X^{2}-(x+y)^{N}X+x^{N}y^{N}~({\rm morphisme~de~Frobenius}) \\
									 &=& X^{2}-(x+y)^{N}X+(xy)^{N}~({\rm commutativit\acute{e}~de}~K) \\
									 &=& X^{2}-\overline{k}^{N}X+\overline{1}~(x~{\rm et}~y~{\rm racines~de}~\chi) \\
									 &=& X^{2}-\overline{k}X+\overline{1}. \\
\end{eqnarray*}
Donc, $x^{N}$ et $y^{N}$ sont des racines de $\chi$. De plus, $x^{N} \neq x$. En effet, supposons par l'absurde que $x^{N}=x$. Dans ce cas, $x^{N-1}=\overline{1}$ ($x \neq \overline{0}$ car $xy=\overline{1}$). Or, le polynôme $Q(X)=X^{N-1}-\overline{1}$ a au plus $N-1$ racines dans $K$ et les éléments non nuls de $\mathbb{Z}/N\mathbb{Z}$ sont des racines. Donc, les seules racines de $Q$ sont les éléments non nuls de $\mathbb{Z}/N\mathbb{Z}$. On a donc, $x \in \mathbb{Z}/N\mathbb{Z}$. Ceci est absurde puisque $\chi$ n'a pas de racine dans ce corps.
\\
\\Ainsi, $x^{N}=y$ et donc $x^{N+1}=xy=\overline{1}$. En procédant comme dans le cas précédent, on obtient $\overline{x}^{\frac{N+1}{2}}=\overline{y}^{\frac{N+1}{2}}=\pm \overline{1}$ ce qui implique que la taille de la solution $\overline{k}$-monomiale minimale de \eqref{p} divise $\frac{N+1}{2}.$
\end{itemize}

\end{proof}

\noindent On donne en annexe (voir annexe \ref{A}) les valeurs des tailles des solutions monomiales minimales pour les nombres premiers compris entre 11 et 47.

\begin{remark}

{\rm Il existe des cas où la taille des solutions monomiales minimales est égale à $\frac{N-1}{2}$ ou $\frac{N+1}{2}$ (voir annexe \ref{A}).}

\end{remark}

On peut également obtenir la taille précise de certaines solutions monomiales minimales lorsque $N$ et $\overline{k}$ vérifient certaines propriétés (voir sous-partie suivante).

\subsection{Solutions $\overline{\frac{N}{l}}$-monomiales minimales}

\subsubsection{Preuve du théorème \ref{24}}

On suppose que $N$ est un entier pair supérieur à 4.
\\
\\i)~Si $4 \mid N$, alors $\overline{\frac{N}{2}}\overline{\frac{N}{2}}=\overline{\frac{N^{2}}{4}}=\overline{N\frac{N}{4}}=\overline{0}$ et donc $(\overline{-\frac{N}{2}},\overline{\frac{N}{2}},\overline{\frac{N}{2}},\overline{-\frac{N}{2}})=(\overline{\frac{N}{2}},\overline{\frac{N}{2}},\overline{\frac{N}{2}},\overline{\frac{N}{2}})$ est solution de \eqref{p} (voir proposition \ref{32}). La taille de la solution $\overline{\frac{N}{2}}$-monomiale minimale divise 4 (puisque c'est l'ordre de $M_{1}(\overline{\frac{N}{2}})$ dans le groupe $PSL_{2}(\mathbb{Z}/N\mathbb{Z})$). \eqref{p} n'a pas de solution de taille 1. Comme $\overline{\frac{N}{2}} \neq \overline{0}$, $(\overline{\frac{N}{2}},\overline{\frac{N}{2}})$ n'est pas solution et donc la solution $\overline{\frac{N}{2}}$-monomiale minimale est de taille 4. Comme, $\overline{\frac{N}{2}} \neq \overline{\pm 1}$, la solution $\overline{\frac{N}{2}}$-monomiale minimale est irréductible. En effet, une solution réductible de taille 4 est la somme de deux solutions de taille 3 et donc contient nécessairement $\pm \overline{1}$.
\\
\\ii)~Si $4 \nmid N$, on note $K=\frac{N}{2}$. En particulier, on a $K$ impair et $\overline{2K}=\overline{0}$. On a

\[M_{6}(\overline{K},\overline{K},\overline{K},\overline{K},\overline{K},\overline{K}) = \begin{pmatrix}
   \overline{K^{6}-5K^{4}+6K^{2}-1}   & \overline{-K^{5}+4K^{3}-3K} \\
   \overline{K^{5}-4K^{3}+3K} & \overline{-K^{4}+3K^{2}-1}
\end{pmatrix}.\]
Or, 
\begin{eqnarray*} 
\overline{K^{6}-5K^{4}+6K^{2}} &=& \overline{K^{6}-K^{4}-4K^{4}+3\times 2K^{2}} \\
                               &=& \overline{K^{6}-K^{4}} \\
															 &=& \overline{K^{4}(K^{2}-1)}. \\
\end{eqnarray*}

De plus, $K^{2}$ est impair (produit de deux entiers impairs) donc $(K^{2}-1)$ est pair. Ainsi, il existe un entier $j$ tel que $(K^{2}-1)=2j$. Donc, \[\overline{K^{6}-5K^{4}+6K^{2}}=\overline{2jK^{4}}=\overline{jK^{3}(2K)}=\overline{0}.\]

\noindent De même, on a
\begin{eqnarray*} 
\overline{-K^{5}+4K^{3}-3K} &=& \overline{-K^{5}+2K(2K^{2})-K-2K} \\
                            &=& \overline{-K^{5}-K} \\
														&=& \overline{-K(1+K^{4})}. \\
\end{eqnarray*}
Or, $K^{4}$ est impair (produit de quatre entiers impairs) donc $1+K^{4}$ est pair. Il existe un entier $j'$ tel que $1+K^{4}=2j'$. Ainsi, \[\overline{-K^{5}+4K^{3}-3K}=\overline{-K(1+K^{4})}=\overline{-j'N}=\overline{0}.\]
\\On procède de façon analogue pour $\overline{-K^{4}+3K^{2}}$. On a $\overline{-K^{4}+3K^{2}}=\overline{K^{2}(1-K^{2})}$. Or, $K^{2}$ est impair (produit de deux entiers impairs) donc $(1-K^{2})$ est pair. Il existe un entier $j''$ tel que $(1-K^{2})=2j''$. Donc, \[\overline{-K^{4}+3K^{2}}=\overline{K^{2}(1-K^{2})}=\overline{2j''K^{2}}=\overline{Nj''K}=\overline{0}.\]
\\Ainsi, $M_{6}(\overline{\frac{N}{2}},\overline{\frac{N}{2}},\overline{\frac{N}{2}},\overline{\frac{N}{2}},\overline{\frac{N}{2}},\overline{\frac{N}{2}})=-Id$. Il en découle que la taille de la solution $\overline{\frac{N}{2}}$-monomiale minimale divise 6 (puisque c'est l'ordre de $M_{1}(\overline{\frac{N}{2}})$ dans le groupe $PSL_{2}(\mathbb{Z}/N\mathbb{Z})$) c'est-à-dire que celle-ci est égale à 1, 2, 3 ou 6.
\\
\\ \eqref{p} n'a pas de solution de taille 1. $\overline{\frac{N}{2}} \neq \overline{0}$, donc, $(\overline{\frac{N}{2}},\overline{\frac{N}{2}})$ n'est pas solution. $\overline{\frac{N}{2}} \neq \overline{\pm 1}$, donc, $(\overline{\frac{N}{2}},\overline{\frac{N}{2}},\overline{\frac{N}{2}})$ n'est pas solution. On en déduit que la solution $\overline{\frac{N}{2}}$-monomiale minimale de \eqref{p} est de taille 6. 
\\
\\Si celle-ci est réductible alors elle est la somme de deux solutions de taille 4 (puisqu'elle ne contient pas $\pm \overline{1}$). Dans ce cas, \eqref{p} a une solution de la forme $(\overline{a},\overline{\frac{N}{2}},\overline{\frac{N}{2}},\overline{a})$ avec $\overline{a} \neq \overline{\frac{N}{2}}$ (sinon la solution minimale serait de taille 4). Donc, on a $\overline{\frac{N^{2}}{4}}=\overline{0}$ et $\overline{a}=\overline{-\frac{N}{2}}$. Ainsi, Il existe un entier $l$ tel que $K^{2}=2lK$. Donc, $K^{2}$ est pair ce qui implique $K$ pair. Ce qui est absurde. 
\\
\\Donc, la solution $\overline{\frac{N}{2}}$-monomiale minimale est irréductible de taille 6.

\qed

\begin{remark}

{\rm Si $N=2$, alors la solution $\overline{\frac{N}{2}}$-monomiale minimale est la solution $\overline{1}$-monomiale minimale qui est irréductible de taille 3.}

\end{remark}

\subsubsection{Généralisations partielles}

On cherche à généraliser le théorème \ref{24} pour des diviseurs de $N$ différents de 2. On commence par le résultat suivant :  

\begin{proposition}
\label{35}

Si $l^{2} \mid N$ avec $l \geq 2$ alors $(\overline{\frac{N}{l}},\ldots,\overline{\frac{N}{l}}) \in (\mathbb{Z}/N\mathbb{Z})^{2l}$ est solution de \eqref{p}.

\end{proposition}

\begin{proof}

On a 
\begin{eqnarray*}
M_{2l}(\overline{\frac{N}{l}},\ldots,\overline{\frac{N}{l}}) &=& (\begin{pmatrix}
   \overline{\frac{N}{l}}   & \overline{-1} \\
   \overline{1} & \overline{0}
\end{pmatrix}\begin{pmatrix}
   \overline{\frac{N}{l}}   & \overline{-1} \\
   \overline{1} & \overline{0}
\end{pmatrix})^{l} \\
                                         &=& \begin{pmatrix}
   \overline{\frac{N^{2}}{l^{2}}-1}   & \overline{-\frac{N}{l}} \\
   \overline{\frac{N}{l}} & \overline{-1}
\end{pmatrix}^{l} \\
                                         &=& \begin{pmatrix}
   \overline{N\frac{N}{l^{2}}-1}   & \overline{-\frac{N}{l}} \\
   \overline{\frac{N}{l}} & \overline{-1}
\end{pmatrix}^{l}~{\rm car}~l^{2} \mid N \\
                                         &=& \begin{pmatrix}
   \overline{-1}   & \overline{-\frac{N}{l}} \\
   \overline{\frac{N}{l}} & \overline{-1}
\end{pmatrix}^{l} \\
                                         &=& \overline{(-Id+\frac{N}{l}S)^{l}} \\
																				 &=& \overline{ \sum_{k=0}^{l} \binom{l}{k} (-1)^{l-k}(\frac{N}{l}S)^{k}}~({\rm bin\hat{o}me~de~Newton}) \\
																				 &=& \overline{ (-1)^{l}\binom{l}{0}Id+(-1)^{l-1}\binom{l}{1}\frac{N}{l}S+\sum_{k=2}^{l} \binom{l}{k} (-1)^{l-k}\frac{N^{k}}{l^{k}}S^{k}} \\
																				 &=& \overline{ (-1)^{l}Id+(-1)^{l-1}NS+\sum_{k=2}^{l} \binom{l}{k} (-1)^{l-k}N\frac{N}{l^{2}}\frac{N^{k-2}}{l^{k-2}}S^{k}} \\
																				 &=& \overline{(-1)}^{l}Id. \\
\end{eqnarray*}

\noindent L'avant dernière égalité est valide car $\frac{N}{l^{2}}$ et $\frac{N^{k-2}}{l^{k-2}}$ sont des entiers puisque $l^{2} \mid N$ et $l^{k-2} \mid N^{k-2}$.

\end{proof}

\begin{remark}

{\rm Si $l=1$ alors  $(\overline{\frac{N}{l}},\ldots,\overline{\frac{N}{l}})=(\overline{0},\overline{0}) \in (\mathbb{Z}/N\mathbb{Z})^{2}$ est aussi solution de \eqref{p}.}

\end{remark} 

\noindent Le résultat ci-dessous résout la question de l'irréductibilité de ces solutions.

\begin{proposition}
\label{36}

Soit $l^{2} \mid N$ avec $l \geq 2$. $(\overline{\frac{N}{l}},\ldots,\overline{\frac{N}{l}}) \in (\mathbb{Z}/N\mathbb{Z})^{2l}$ est une solution irréductible de \eqref{p} si et seulement si $l=2$.

\end{proposition}

\begin{proof}

Si $l=2$ alors la solution est irréductible (voir théorème \ref{24}). 
\\
\\Si $l \geq 3$ alors la solution n'est pas irréductible car $(\overline{-\frac{N}{l}},\overline{\frac{N}{l}},\overline{\frac{N}{l}},\overline{-\frac{N}{l}})$ est solution de \eqref{p} puisque $\overline{\frac{N^{2}}{l^{2}}}=\overline{N\frac{N}{l^{2}}}=\overline{0}$ (voir proposition \ref{32}) et \[(\overline{\frac{N}{l}},\ldots,\overline{\frac{N}{l}})=(\overline{2\frac{N}{l}},\overline{\frac{N}{l}},\ldots,\overline{\frac{N}{l}},\overline{2\frac{N}{l}}) \oplus (\overline{-\frac{N}{l}},\overline{\frac{N}{l}},\overline{\frac{N}{l}},\overline{-\frac{N}{l}}).\]  $(\overline{2\frac{N}{l}},\overline{\frac{N}{l}},\ldots,\overline{\frac{N}{l}},\overline{2\frac{N}{l}})$ est de taille $2l-2 \geq 4 >3$.

\end{proof}

\begin{remark}

{\rm Ces propositions généralisent les propriétés 3.10 et 3.11 de \cite{M}.}

\end{remark}

Avant de continuer, on a besoin des résultats suivants qui permettent d'obtenir une expression de $M_{n}(a_{1},\ldots,a_{n})$ :
\\
\\Soit $x$ un réel, on note $E[x]$ la partie entière de $x$. On pose $K_{-1}=0$, $K_{0}=1$ et on note pour $i \geq 1$ \[K_i(a_{1},\ldots,a_{i})=
\left|
\begin{array}{cccccc}
a_1&1&&&\\[4pt]
1&a_{2}&1&&\\[4pt]
&\ddots&\ddots&\!\!\ddots&\\[4pt]
&&1&a_{i-1}&\!\!\!\!\!1\\[4pt]
&&&\!\!\!\!\!1&\!\!\!\!a_{i}
\end{array}
\right|.\] $K_{i}(a_{1},\ldots,a_{i})$ est le continuant de $a_{1},\ldots,a_{i}$. On dispose de l'égalité suivante (voir \cite{CO,MO}) : \[M_{n}(a_{1},\ldots,a_{n})=\begin{pmatrix}
    K_{n}(a_{1},\ldots,a_{n}) & -K_{n-1}(a_{2},\ldots,a_{n}) \\
    K_{n-1}(a_{1},\ldots,a_{n-1})  & -K_{n-2}(a_{2},\ldots,a_{n-1}) 
   \end{pmatrix}.\]
De plus, on dispose de l'expression classique ci-dessous :

\begin{lemma}

Soit $n \geq 0$, $K_{n}(x,\ldots,x)=\sum_{k=0}^{E[\frac{n}{2}]} (-1)^{k}\binom{n-k}{k}x^{n-2k}$.

\end{lemma}

\begin{proof}

Cela se prouve par récurrence sur $n$. En effet, la  formule est vraie pour $n=0$ et pour $n=1$. Supposons qu'il existe un entier positif $n$ tel que la formule est vraie pour $n$ et $n-1$. On suppose $n$ pair. En développant le déterminant définissant $K_{n+1}(x,\ldots,x)$ suivant la première colonne, on obtient :
\begin{eqnarray*}
K_{n+1}(x,\ldots,x) &=& xK_{n}(x,\ldots,x)-K_{n-1}(x,\ldots,x) \\
                    &=& \sum_{k=0}^{E[\frac{n}{2}]} (-1)^{k}\binom{n-k}{k}x^{n+1-2k}-\sum_{k=0}^{E[\frac{n-1}{2}]} (-1)^{k}\binom{n-1-k}{k}x^{n-1-2k} \\
										&=&  \sum_{k=0}^{\frac{n}{2}} (-1)^{k}\binom{n-k}{k}x^{n+1-2k}-\sum_{k=0}^{\frac{n}{2}-1} (-1)^{k}\binom{n-1-k}{k}x^{n-1-2k}~({\rm car}~n~{\rm est~pair}) \\
										&=&  \sum_{k=0}^{\frac{n}{2}} (-1)^{k}\binom{n-k}{k}x^{n+1-2k}-\sum_{l=1}^{\frac{n}{2}} (-1)^{l-1}\binom{n-l}{l-1}x^{n+1-2l} \\ 
										&=&  \sum_{k=0}^{\frac{n}{2}} (-1)^{k}\binom{n-k}{k}x^{n+1-2k}+\sum_{l=1}^{\frac{n}{2}} (-1)^{l}\binom{n-l}{l-1}x^{n+1-2l} \\
										&=&  x^{n+1}+\sum_{k=1}^{\frac{n}{2}} (-1)^{k}(\binom{n-k}{k}+\binom{n-k}{k-1})x^{n+1-2k} \\
										&=&  x^{n+1}+\sum_{k=1}^{\frac{n}{2}} (-1)^{k}\binom{n+1-k}{k}x^{n+1-2k}~({\rm triangle~de~Pascal}) \\
										&=&  \sum_{k=0}^{\frac{n}{2}} (-1)^{k}\binom{n+1-k}{k}x^{n+1-2k}. \\
\end{eqnarray*}

\noindent On procède de façon analogue si $n$ est impair. Cela prouve la formule par récurrence.

\end{proof}

\begin{proposition}
\label{37}

Si $p^{2} \mid N$ avec $p$ premier alors $(\overline{\frac{N}{p}},\ldots,\overline{\frac{N}{p}}) \in (\mathbb{Z}/N\mathbb{Z})^{2p}$ est une solution monomiale minimale de \eqref{p}.

\end{proposition}

\begin{proof}

Si $p=2$ alors le résultat est vrai (voir théorème \ref{24}). On suppose maintenant $p \geq 3$. Par ce qui précède, $(\overline{\frac{N}{p}},\ldots,\overline{\frac{N}{p}}) \in (\mathbb{Z}/N\mathbb{Z})^{2p}$ est une solution monomiale de \eqref{p}. Ainsi, la taille de la solution $\overline{\frac{N}{p}}$-monomiale minimale divise $2p$. Donc, celle-ci est égale à 1, 2, $p$ ou $2p$. \eqref{p} n'a pas de solution de taille 1, et $\overline{\frac{N}{p}} \neq \overline{0}$ (sinon $N \mid \frac{N}{p}$ et donc $\frac{N}{p} \geq N$) donc la solution $\overline{\frac{N}{p}}$-monomiale minimale n'est pas de taille 2. Supposons par l'absurde que la solution $\overline{\frac{N}{p}}$-monomiale minimale est de taille $p$.
\\
\\ Il existe $\epsilon$ dans $\{\pm 1\}$ tel que \[\overline{\epsilon}Id=M_{p}(\overline{\frac{N}{p}},\ldots,\overline{\frac{N}{p}})=\begin{pmatrix}
    K_{p}(\overline{\frac{N}{p}},\ldots,\overline{\frac{N}{p}}) & -K_{p-1}(\overline{\frac{N}{p}},\ldots,\overline{\frac{N}{p}}) \\
    K_{p-1}(\overline{\frac{N}{p}},\ldots,\overline{\frac{N}{p}})  & -K_{p-2}(\overline{\frac{N}{p}},\ldots,\overline{\frac{N}{p}}) 
   \end{pmatrix}.\]
	
\noindent Donc, $K_{p-1}(\overline{\frac{N}{p}},\ldots,\overline{\frac{N}{p}})=\overline{0}$. Notons $K=K_{p-1}(\overline{\frac{N}{p}},\ldots,\overline{\frac{N}{p}})$. On a, par le lemme précédent,
\begin{eqnarray*}
K &=& \overline{\sum_{k=0}^{E[\frac{p-1}{2}]} (-1)^{k}\binom{p-1-k}{k}(\frac{N}{p})^{p-1-2k}} \\
  &=& \overline{\sum_{k=0}^{\frac{p-1}{2}} (-1)^{k}\binom{p-1-k}{k}(\frac{N}{p})^{p-1-2k}}~({\rm car}~p-1~{\rm est~pair}) \\
	&=& \overline{\sum_{k=0}^{\frac{p-1}{2}-1} (-1)^{k}\binom{p-1-k}{k}\frac{N^{p-1-2k}}{p^{p-1-2k}}+ (-1)^{\frac{p-1}{2}}\binom{\frac{p-1}{2}}{\frac{p-1}{2}}}  \\
	&=& \overline{(-1)^{\frac{p-1}{2}} + \sum_{k=0}^{\frac{p-1}{2}-1} (-1)^{k}\binom{p-1-k}{k}N\frac{N}{p^{2}}\frac{N^{p-1-2k-2}}{p^{p-1-2k-2}}}~(p-1-2k-2 \geq 0~{\rm et}~p^{2} \mid N) \\
	&=& \overline{(-1)^{\frac{p-1}{2}}} \\
	& \neq & \overline{0}. \\
\end{eqnarray*}

\noindent Ceci est absurde. Donc, la solution $\overline{\frac{N}{p}}$-monomiale minimale est de taille $2p$.

\end{proof}

\subsection{Réductibilité dans le cas $N=l^{n}$}

On se place dans le cas où $N=l^{n}$ avec $n$ et $l$ supérieurs à 2. On sait que le $2l^{n-1}$-uplet $(\overline{l},\ldots,\overline{l})$ d'éléments de $\mathbb{Z}/N\mathbb{Z}$ est une solution de \eqref{p} (voir \cite{M} proposition 3.14). Cependant, la question de l’irréductibilité potentielle de cette solution reste ouverte. On se propose ici de répondre à cette dernière en démontrant le résultat suivant :
		
\begin{theorem}
\label{310}

Si $N=l^{n}$ avec $l,n \geq 2$ alors $(\overline{l},\ldots,\overline{l}) \in (\mathbb{Z}/N\mathbb{Z})^{2l^{n-1}}$ est une solution irréductible de \eqref{p} si et seulement si $l=2$.

\end{theorem}

\noindent Pour cela on a besoin de plusieurs résultats intermédiaires. On commence par le résultat classique suivant :

\begin{lemma}
\label{3101}

Soient $n \in \mathbb{N}^{*}$ et $k \in [\![1;n]\!]$, $\frac{n}{{\rm pgcd}(n,k)}~{\rm divise}~\binom{n}{k}$.

\end{lemma}

\begin{proof}

\[{n \choose k}=\frac{n}{k}{n-1 \choose k-1}=\frac{\frac{n}{{\rm pgcd}(n,k)}}{\frac{k}{{\rm pgcd}(n,k)}}{n-1 \choose k-1}.\] Donc, comme ${n \choose k} \in \mathbb{N}^{*}$, on a $\frac{k}{{\rm pgcd}(n,k)}~{\rm divise}~\frac{n}{{\rm pgcd}(n,k)}{n-1 \choose k-1}$. Comme $\frac{k}{{\rm pgcd}(n,k)}$ et $\frac{n}{{\rm pgcd}(n,k)}$ sont premiers entre eux, on a, par le lemme de Gauss, $\frac{k}{{\rm pgcd}(n,k)}~{\rm divise}~{n-1 \choose k-1}$. Donc, $\frac{n}{{\rm pgcd}(n,k)}~{\rm divise}~{n \choose k}$.

\end{proof}

\begin{lemma}
\label{311}

Soient $(n,l) \in (\mathbb{N}^{*})^{2}$, $n \geq 3$, $l \geq 2$ et $j \in [\![2;n-1]\!]$. On a $l^{n-j}~{\rm divise}~\binom{2l^{n-2}}{j}$.

\end{lemma}

\begin{proof}

Si $j=2$ alors \[\binom{2l^{n-2}}{j}=\frac{2l^{n-2}(2l^{n-2}-1)}{2}=l^{n-2}(2l^{n-2}-1)\] et donc le résultat est vrai. On peut ainsi supposer $n \geq 4$ et $j \geq 3$. On a par le lemme précédent
\[\frac{2l^{n-2}}{{\rm pgcd}(2l^{n-2},j)}~{\rm divise}~\binom{2l^{n-2}}{j}.\] Notons $l=p_{1}^{\alpha_{1}} \ldots p_{r}^{\alpha_{r}}$ la décomposition de $l$ en facteurs premiers. On a deux cas :
\begin{itemize}
\item $l$ est impair. Dans ce cas, pour tout $i$ appartenant à $[\![1;r]\!]$, $p_{i} \neq 2$. $\exists(\beta_{1},\ldots,\beta_{r},a) \in \mathbb{N}^{r+1}$ tels que ${\rm pgcd}(2l^{n-2},j)=2^{a}p_{1}^{\beta_{1}} \ldots p_{r}^{\beta_{r}}$. Si $j$ est pair alors $a=1$ et si $j$ est impair alors $a=0$. 
\\
\\Montrons que pour tout $i$ dans $[\![1;r]\!]$, $\beta_{i} \leq \alpha_{i}(j-2)$. Supposons par l'absurde qu'il existe $i$ dans $[\![1;r]\!]$ tel que $\beta_{i} > \alpha_{i}(j-2)$. Par récurrence, on montre que si $j \geq 3$ on a $p_{i}^{j-2} \geq j$ (car $p_{i} > 2$). On a \[p_{i}^{\beta_{i}}> p_{i}^{\alpha_{i}(j-2)} \geq p_{i}^{j-2} \geq j.\] Donc, ${\rm pgcd}(2l^{n-2},j) >j$ ce qui est absurde.
\\
\\Ainsi, pour tout $i$ appartenant à $[\![1;r]\!]$, $\beta_{i} \leq \alpha_{i}(j-2)$. De plus, on a $a=1$ si $j$ est pair et $a=0$ si $j$ est impair. On en déduit que $l^{n-j}$ divise $\frac{2l^{n-2}}{{\rm pgcd}(2l^{n-2},j)}$. Donc, $l^{n-j}$ divise $\binom{2l^{n-2}}{j}$.
\\
\item $l$ est pair. Dans ce cas, on peut supposer $p_{1}=2$, et donc $p_{j} >2$ pour $j$ dans $[\![2;r]\!]$. 
\\$\exists(\beta_{1},\ldots,\beta_{r}) \in \mathbb{N}^{r}$ tels que ${\rm pgcd}(2l^{n-2},j)=p_{1}^{\beta_{1}} \ldots p_{r}^{\beta_{r}}$.
\\
\\On montre, en procédant comme dans le premier cas, que pour tout $i$ dans $[\![2;r]\!]$, $\beta_{i} \leq \alpha_{i}(j-2)$. Montrons que $\beta_{1} \leq \alpha_{1}(j-2)+1$. Si $\beta_{1} > \alpha_{1}(j-2)+1$ alors \[p_{1}^{\beta_{1}} > p_{1}^{\alpha_{1}(j-2)+1} \geq p_{1}^{j-1} \geq j.\] Donc, ${\rm pgcd}(2l^{n-2},j) >j$ ce qui est absurde.
\\
\\On en déduit que pour tout $i$ dans $[\![2;r]\!]$, $\beta_{i} \leq \alpha_{i}(j-2)$ et $\beta_{1} \leq \alpha_{1}(j-2)+1$. Ainsi, $l^{n-j}$ divise $\frac{2l^{n-2}}{{\rm pgcd}(2l^{n-2},j)}$. Donc, $l^{n-j}~{\rm divise}~\binom{2l^{n-2}}{j}$.
\\

\end{itemize}

\noindent Donc, le résultat est vrai.

\end{proof}

\begin{lemma}
\label{312}

Si $N=l^{n}$ avec $l>2$ et $n \geq 3$ alors $(\overline{2l^{n-1}},\overline{l},\ldots,\overline{l},\overline{2l^{n-1}}) \in (\mathbb{Z}/N\mathbb{Z})^{2l^{n-1}-4l^{n-2}+2}$ est une solution de \eqref{p}.

\end{lemma}

\begin{proof}

On a $2l^{n-1}-4l^{n-2}+2=2l^{n-2}(l-2)+2 \geq 2l^{n-2}+2 \geq 2l+2 \geq 8$.
\\
\\$(\overline{2l^{n-1}},\overline{l},\ldots,\overline{l},\overline{2l^{n-1}}) \sim (\overline{2l^{n-1}},\overline{2l^{n-1}},\overline{l},\ldots,\overline{l})$. Donc, $(\overline{2l^{n-1}},\overline{l},\ldots,\overline{l},\overline{2l^{n-1}})$ est solution de \eqref{p} si et seulement si $(\overline{2l^{n-1}},\overline{2l^{n-1}},\overline{l},\ldots,\overline{l})$ l'est aussi. On a \[M=M_{2l^{n-1}-4l^{n-2}+2}(\overline{2l^{n-1}},\overline{2l^{n-1}},\overline{l},\ldots,\overline{l})=M_{2l^{n-1}-4l^{n-2}}(\overline{l},\ldots,\overline{l})M_{2}(\overline{2l^{n-1}},\overline{2l^{n-1}}).\]

\begin{eqnarray*}
M_{2l^{n-1}-4l^{n-2}}(\overline{l},\ldots,\overline{l}) &=& M_{2l^{n-1}}(\overline{l},\ldots,\overline{l})M_{4l^{n-2}}(\overline{l},\ldots,\overline{l})^{-1} \\ 
                                                        &=& \overline{(-1)^{l^{n-1}}}M_{4l^{n-2}}(\overline{l},\ldots,\overline{l})^{-1}~\\
                                                        &=& \overline{(-1)^{l^{n-1}}}(\begin{pmatrix}
   \overline{0}   & \overline{1} \\
   \overline{-1} & \overline{l}
\end{pmatrix})^{4l^{n-2}}\\
                                                        &=& \overline{(-1)^{l^{n-1}}}(\begin{pmatrix}
   \overline{0}   & \overline{1} \\
   \overline{-1} & \overline{l}
\end{pmatrix}^{2})^{2l^{n-2}} \\
                                                        &=& \overline{(-1)^{l^{n-1}}}\begin{pmatrix}
   \overline{-1}   & \overline{l} \\
   \overline{-l} & \overline{-1+l^{2}}
\end{pmatrix}^{2l^{n-2}} \\
                                                        &=& \overline{(-1)^{l^{n-1}}}\overline{(-Id+l\begin{pmatrix}
   0   & 1 \\
   -1 & l
\end{pmatrix})^{2l^{n-2}}} \\
																				                &=& \overline{ \sum_{k=0}^{2l^{n-2}} \binom{2l^{n-2}}{k} (-1)^{2l^{n-2}+l^{n-1}-k}l^{k}\begin{pmatrix}
   0   & 1 \\
   -1 & l
\end{pmatrix}^{k}}~{\rm (bin\hat{o}me~de~Newton)} \\
																				                &=& \overline{ \sum_{k=0}^{n-1} \binom{2l^{n-2}}{k} (-1)^{l^{n-2}(l+2)-k}l^{k}\begin{pmatrix}
   0   & 1 \\
   -1 & l
\end{pmatrix}^{k}} \\
																				                &=& \overline{(-1)^{l^{n-2}(l+2)}Id+(-1)^{l^{n-2}(l+2)-1}(2l^{n-2})l\begin{pmatrix}
   0   & 1 \\
   -1 & l
\end{pmatrix}}. \\
\end{eqnarray*}

\noindent La dernière égalité est valide car, par le lemme précédent, $l^{n-k}$ divise $\binom{2l^{n-2}}{k}$ pour $2 \leq k \leq n-1$.
\\
\\ \noindent De plus, $M_{2}(\overline{2l^{n-1}},\overline{2l^{n-1}})=\begin{pmatrix}
   \overline{-1}   & \overline{-2l^{n-1}} \\
   \overline{2l^{n-1}} & \overline{-1}
\end{pmatrix}.$
\\
\\Donc, on a
\begin{eqnarray*}
M &=& M_{2l^{n-1}-4l^{n-2}}(\overline{l},\ldots,\overline{l})M_{2}(\overline{2l^{n-1}},\overline{2l^{n-1}})\\
  &=& (\overline{(-1)^{l^{n-2}(l+2)}Id}+\overline{(-1)^{l^{n-2}(l+2)-1}(2l^{n-2})l}\begin{pmatrix}
   \overline{0}   & \overline{1} \\
   \overline{-1} & \overline{l}
\end{pmatrix})\begin{pmatrix}
   \overline{-1}   & \overline{-2l^{n-1}} \\
   \overline{2l^{n-1}} & \overline{-1}
\end{pmatrix} \\
 &=& \overline{(-1)^{l^{n-2}(l+2)}}\begin{pmatrix}
   \overline{-1}   & \overline{-2l^{n-1}} \\
   \overline{2l^{n-1}} & \overline{-1}
\end{pmatrix}+\overline{(-1)^{l^{n-2}(l+2)-1}(2l^{n-1})}\begin{pmatrix}
   \overline{2l^{n-1}}   & \overline{-1} \\
   \overline{1} & \overline{2l^{n-1}-l}
\end{pmatrix}\\
 &=& \overline{(-1)^{l^{n-2}(l+2)}}(\begin{pmatrix}
   \overline{-1}   & \overline{-2l^{n-1}} \\
   \overline{2l^{n-1}} & \overline{-1}
\end{pmatrix}-\begin{pmatrix}
   \overline{0}   & \overline{-2l^{n-1}} \\
   \overline{2l^{n-1}} & \overline{0}
\end{pmatrix}) \\
 &=& \overline{(-1)^{l^{n-2}(l+2)+1}Id}.
\end{eqnarray*}

\end{proof}

\noindent On peut maintenant démontrer le résultat principal de la section.

\begin{proof}[Démonstration du Théorème \ref{310}]

Soit $N=l^{n}$ avec $l \geq 2$ et $n \geq 2$. $(\overline{l},\ldots,\overline{l}) \in (\mathbb{Z}/N\mathbb{Z})^{2l^{n-1}}$ est solution de \eqref{p} (voir \cite{M} proposition 3.14).
\\
\\Si $n=2$ alors le résultat est vrai (voir proposition \ref{36}).
\\
\\Supposons $n \geq 3$. Si $l=2$ alors la solution est irréductible (voir théorème \ref{31}). Supposons $l>2$. On a \[(\overline{l},\ldots,\overline{l})=(\overline{l-2l^{n-1}},\overline{l},\ldots,\overline{l},\overline{l-2l^{n-1}}) \oplus (\overline{2l^{n-1}},\overline{l},\ldots,\overline{l},\overline{2l^{n-1}})\] avec $(\overline{2l^{n-1}},\overline{l},\ldots,\overline{l},\overline{2l^{n-1}}) \in (\mathbb{Z}/N\mathbb{Z})^{2l^{n-1}-4l^{n-2}+2}$ et $(\overline{l-2l^{n-1}},\overline{l},\ldots,\overline{l},\overline{l-2l^{n-1}}) \in (\mathbb{Z}/N\mathbb{Z})^{4l^{n-2}}$. 
\\
\\De plus, par le lemme précédent, $(\overline{2l^{n-1}},\overline{l},\ldots,\overline{l},\overline{2l^{n-1}}) \in (\mathbb{Z}/N\mathbb{Z})^{2l^{n-1}-4l^{n-2}+2}$ est une solution de \eqref{p} de taille supérieure à 3. $(\overline{l-2l^{n-1}},\overline{l},\ldots,\overline{l},\overline{l-2l^{n-1}}) \in (\mathbb{Z}/N\mathbb{Z})^{4l^{n-2}}$ est de taille supérieure à 4. Donc, $(\overline{l},\ldots,\overline{l}) \in (\mathbb{Z}/N\mathbb{Z})^{2l^{n-1}}$ est une solution réductible de \eqref{p}.

\end{proof}

\begin{remark}

{\rm Notons que $(\overline{l-2l^{n-1}},\overline{l},\ldots,\overline{l},\overline{l-2l^{n-1}}) \in (\mathbb{Z}/N\mathbb{Z})^{4l^{n-2}}$ est aussi une solution de \eqref{p}.}

\end{remark}

\section{Solutions dynomiales minimales}\label{DM}

On s'intéresse dans cette section au concept de solutions dynomiales minimales défini dans la section \ref{RP} en démontrant notamment le théorème \ref{26}.

\subsection{Résultats préliminaires}

Pour effectuer la preuve du théorème \ref{26}, on a besoin de plusieurs résultats intermédiaires. On commence par le lemme ci-dessous  :

\begin{lemma}
\label{41}

Soit $n \in \mathbb{N} \cup \{-1\}$, $K_{n}(x_{1},\ldots,x_{n})=(-1)^{n}K_{n}(-x_{1},\ldots,-x_{n})$.

\end{lemma}

\begin{proof}

On raisonne par récurrence sur $n$.
\\
\\ $K_{0}=1$ et $K_{-1}=0$ donc le résultat est vrai pour $n=-1$ et pour $n=0$.
\\
\\Supposons qu'il existe $n \in \mathbb{N}$ tel que la formule est vraie au rang $n$ et $n-1$. On a :
\begin{eqnarray*}
K_{n+1}(x_{1},\ldots,x_{n+1}) &=& x_{1}K_{n}(x_{2},\ldots,x_{n+1})-K_{n-1}(x_{3},\ldots,x_{n+1}) \\
                            &=& (-1)^{n}x_{1}K_{n}(-x_{2},\ldots,-x_{n+1})-(-1)^{n-1}K_{n-1}(-x_{3},\ldots,-x_{n+1}) \\
														&=& (-1)^{n-1}(-x_{1}K_{n}(-x_{2},\ldots,-x_{n+1})-K_{n-1}(-x_{3},\ldots,-x_{n+1})) \\
														&=& (-1)^{n+1}(-x_{1}K_{n}(-x_{2},\ldots,-x_{n+1})-K_{n-1}(-x_{3},\ldots,-x_{n+1})) \\
														&=& (-1)^{n+1}K_{n+1}(-x_{1},\ldots,-x_{n+1}). \\
\end{eqnarray*}		

\noindent Par récurrence, le résultat est vrai.

\end{proof}												

On a également besoin du résultat suivant qui est l'analogue de la proposition 3.15 de \cite{M} pour les solutions dynomiales.

\begin{lemma}
\label{42}

Soient $n \in \mathbb{N}^{*}$, $n \geq 4$ et $(\overline{a},\overline{b},\overline{k}) \in (\mathbb{Z}/N\mathbb{Z})^{3}$. Soit $\alpha \in \{\pm 1\}$.
\\i) Si $(\overline{a},\overline{\alpha k},\overline{-\alpha k},\ldots,\overline{\alpha k},\overline{-\alpha k},\overline{b}) \in (\mathbb{Z}/N\mathbb{Z})^{n}$ est solution de \eqref{p} alors $\overline{a}=\overline{-b}$ et on a \[\overline{0}=\overline{a}(\overline{\alpha k}+\overline{a}).\]
\\
\\ii)Si $(\overline{a},\overline{\alpha k},\overline{-\alpha k},\ldots,\overline{\alpha k},\overline{-\alpha k},\overline{\alpha k},\overline{b}) \in (\mathbb{Z}/N\mathbb{Z})^{n}$ est solution de \eqref{p} alors $\overline{a}=\overline{b}$ et on a \[\overline{2}=\overline{a}(\overline{\alpha k}+\overline{a}).\]
\\
	
\end{lemma}	
	
\begin{proof}

i) Supposons que $n$ est pair. Comme $(\overline{a},\overline{\alpha k},\overline{-\alpha k},\ldots,\overline{\alpha k},\overline{-\alpha k},\overline{b})$ est solution de \eqref{p}, il existe $\epsilon$ dans $\{-1,1\}$ tel que \begin{eqnarray*}
\overline{\epsilon} Id &=& M_{n}(\overline{a},\overline{\alpha k},\overline{-\alpha k},\ldots,\overline{\alpha k},\overline{-\alpha k},\overline{b}) \\
                          &=& \begin{pmatrix}
   K_{n}(\overline{a},\overline{\alpha k},\overline{-\alpha k},\ldots,\overline{\alpha k},\overline{-\alpha k},\overline{b})   & -K_{n-1}(\overline{\alpha k},\overline{-\alpha k},\ldots,\overline{\alpha k},\overline{-\alpha k},\overline{b}) \\
   K_{n-1}(\overline{a},\overline{\alpha k},\overline{-\alpha k},\ldots,\overline{\alpha k},\overline{-\alpha k}) & -K_{n-2}(\overline{\alpha k},\overline{-\alpha k},\ldots,\overline{\alpha k},\overline{-\alpha k})
\end{pmatrix}.\\
\end{eqnarray*}
Ainsi, \[K_{n-1}(\overline{a},\overline{\alpha k},\overline{-\alpha k},\ldots,\overline{\alpha k},\overline{-\alpha k})=-K_{n-1}(\overline{\alpha k},\overline{-\alpha k},\ldots,\overline{\alpha k},\overline{-\alpha k},\overline{b})=\overline{0}\]et \[K_{n-2}(\overline{\alpha k},\overline{-\alpha k},\ldots,\overline{\alpha k},\overline{-\alpha k})=-\overline{\epsilon}.\] Or, 
\begin{eqnarray*}
K_{n-1}(\overline{a},\overline{\alpha k},\overline{-\alpha k},\ldots,\overline{\alpha k},\overline{-\alpha k}) 
 &=& \overline{a}K_{n-2}(\overline{\alpha k},\overline{-\alpha k},\ldots,\overline{\alpha k},\overline{-\alpha k}) \\
 &-& K_{n-3}(\overline{-\alpha k},\overline{\alpha k},\overline{-\alpha k},\ldots,\overline{\alpha k},\overline{-\alpha k}) \\
 &=& \overline{-\epsilon a}-K_{n-3}(\overline{-\alpha k},\overline{\alpha k},\overline{-\alpha k},\ldots,\overline{\alpha k},\overline{-\alpha k}). \\
\end{eqnarray*}
Donc, comme $\overline{\epsilon}^{2}=\overline{1}$, on a \[\overline{a}=\overline{-\epsilon}K_{n-3}(\overline{-\alpha k},\overline{\alpha k},\overline{-\alpha k},\ldots,\overline{\alpha k},\overline{-\alpha k}).\] De même, on a
\begin{eqnarray*}
K_{n-1}(\overline{\alpha k},\overline{-\alpha k},\ldots,\overline{\alpha k},\overline{-\alpha k},\overline{b})
 &=& \overline{b}K_{n-2}(\overline{\alpha k},\overline{-\alpha k},\ldots,\overline{\alpha k},\overline{-\alpha k}) \\
 &-& K_{n-3}(\overline{\alpha k},\overline{-\alpha k},\ldots,\overline{\alpha k},\overline{-\alpha k},\overline{\alpha k}) \\
 &=& \overline{-\epsilon b}-K_{n-3}(\overline{\alpha k},\overline{-\alpha k},\ldots,\overline{\alpha k},\overline{-\alpha k},\overline{\alpha k}). \\
\end{eqnarray*}
Il en découle 
\begin{eqnarray*}
\overline{b} &=& \overline{-\epsilon}K_{n-3}(\overline{\alpha k},\overline{-\alpha k},\ldots,\overline{\alpha k},\overline{-\alpha k},\overline{\alpha k}) \\
             &=& \overline{(-\epsilon)}\overline{(-1)^{n-3}}K_{n-3}(\overline{-\alpha k},\overline{\alpha k},\ldots,\overline{-\alpha k},\overline{\alpha k},\overline{-\alpha k}) \\
						 &=& \overline{-a}~{\rm car}~n~{\rm est~pair}.\\
\end{eqnarray*}
						
\noindent De plus, on dispose des égalités ci-dessous : 
\begin{eqnarray*}
\overline{-\epsilon} &=& K_{n-2}(\overline{\alpha k},\overline{-\alpha k},\ldots,\overline{\alpha k},\overline{-\alpha k}) \\
                     &=& \overline{\alpha k}K_{n-3}(\overline{-\alpha k},\overline{\alpha k},\overline{-\alpha k},\ldots,\overline{\alpha k},\overline{-\alpha k})-K_{n-4}(\overline{\alpha k},\overline{-\alpha k},\ldots,\overline{\alpha k},\overline{-\alpha k})\\
\end{eqnarray*}										
et 
\\
\\$M_{n-2}(\overline{\alpha k},\overline{-\alpha k},\ldots,\overline{\alpha k},\overline{-\alpha k})$
\\
\\ $=\begin{pmatrix}
   K_{n-2}(\overline{\alpha k},\overline{-\alpha k},\ldots,\overline{\alpha k},\overline{-\alpha k})   & -K_{n-3}(\overline{-\alpha k},\overline{\alpha k},\overline{-\alpha k},\ldots,\overline{\alpha k},\overline{-\alpha k}) \\
   K_{n-3}(\overline{\alpha k},\overline{-\alpha k},\ldots,\overline{\alpha k},\overline{-\alpha k},\overline{\alpha k}) & -K_{n-4}(\overline{-\alpha k},\overline{\alpha k},\ldots,\overline{-\alpha k},\overline{\alpha k})
\end{pmatrix}$
\\
\\ $=\begin{pmatrix}
   K_{n-2}(\overline{\alpha k},\overline{-\alpha k},\ldots,\overline{\alpha k},\overline{-\alpha k})   & K_{n-3}(\overline{\alpha k},\overline{-\alpha k},\ldots,\overline{\alpha k},\overline{-\alpha k},\overline{\alpha k}) \\
   K_{n-3}(\overline{\alpha k},\overline{-\alpha k},\ldots,\overline{\alpha k},\overline{-\alpha k},\overline{\alpha k}) & -K_{n-4}(-\overline{\alpha k},\overline{\alpha k},\ldots,\overline{-\alpha k},\overline{\alpha k})
\end{pmatrix} \in SL_{2}(\mathbb{Z}/N\mathbb{Z})$.
\\
\\ \noindent On en déduit l'égalité suivante : \[-K_{n-2}(\overline{\alpha k},\overline{-\alpha k},\ldots,\overline{\alpha k},\overline{-\alpha k})K_{n-4}(\overline{-\alpha k},\overline{\alpha k},\ldots,\overline{-\alpha k},\overline{\alpha k})-K_{n-3}(\overline{\alpha k},\overline{-\alpha k},\ldots,\overline{\alpha k},\overline{-\alpha k},\overline{\alpha k})^{2}=\overline{1}.\] Or, comme $K_{n-2}(\overline{\alpha k},\overline{-\alpha k},\ldots,\overline{\alpha k},\overline{-\alpha k})=\overline{-\epsilon}$, on a \[\overline{\epsilon}K_{n-4}(-\overline{\alpha k},\overline{\alpha k},\ldots,\overline{-\alpha k},\overline{\alpha k})-K_{n-3}(\overline{\alpha k},\overline{-\alpha k},\ldots,\overline{\alpha k},\overline{-\alpha k},\overline{\alpha k})^{2}=\overline{1},\] c'est-à-dire 

\begin{eqnarray*}
K_{n-4}(\overline{-\alpha k},\overline{\alpha k},\ldots,\overline{-\alpha k},\overline{\alpha k}) &=& K_{n-4}(\overline{\alpha k},\overline{-\alpha k},\ldots,\overline{\alpha k},\overline{-\alpha k})~{\rm car}~n-4~{\rm est~pair}\\
                                                                                                  &=& \overline{\epsilon}(\overline{1}+K_{n-3}(\overline{\alpha k},\overline{-\alpha k},\ldots,\overline{\alpha k},\overline{-\alpha k},\overline{\alpha k})^{2}).\\
\end{eqnarray*}	
																																															
\noindent Ainsi, on a

\begin{eqnarray*}
\overline{-\epsilon} &=& \overline{\alpha k}K_{n-3}(\overline{-\alpha k},\overline{\alpha k},\overline{-\alpha k},\ldots,\overline{\alpha k},\overline{-\alpha k})-K_{n-4}(\overline{\alpha k},\overline{-\alpha k},\ldots,\overline{\alpha k},\overline{-\alpha k}) \\
                     &=& -\overline{\alpha k}K_{n-3}(\overline{\alpha k},\overline{-\alpha k},\ldots,\overline{\alpha k},\overline{-\alpha k},\overline{\alpha k})-\overline{\epsilon}(\overline{1}+K_{n-3}(\overline{\alpha k},\overline{-\alpha k},\ldots,\overline{\alpha k},\overline{-\alpha k},\overline{\alpha k})^{2}) \\ 
										 &=& -\overline{\alpha k}K_{n-3}(\overline{\alpha k},\overline{-\alpha k},\ldots,\overline{\alpha k},\overline{-\alpha k},\overline{\alpha k})-\overline{\epsilon}-\overline{\epsilon}K_{n-3}(\overline{\alpha k},\overline{-\alpha k},\ldots,\overline{\alpha k},\overline{-\alpha k},\overline{\alpha k})^{2} \\
										 &=&  \overline{-\alpha \epsilon k a}-\overline{\epsilon}-\overline{\epsilon a^{2}}.\\
\end{eqnarray*}

\noindent Donc, $\overline{0}=\overline{-\alpha \epsilon k a}-\overline{\epsilon a^{2}}=-\overline{\epsilon}\overline{a}(\overline{\alpha k}+\overline{a})$ et donc $\overline{0}=\overline{a}(\overline{\alpha k}+\overline{a})$.
\\
\\ii) Supposons que $n$ est impair. Comme $(\overline{a},\overline{\alpha k},\overline{-\alpha k},\ldots,\overline{\alpha k},\overline{-\alpha k},\overline{\alpha k},\overline{b})$ est solution de \eqref{p}, il existe $\epsilon$ dans $\{-1,1\}$ tel que \begin{eqnarray*}
\overline{\epsilon} Id &=& M_{n}(\overline{a},\overline{\alpha k},\overline{-\alpha k},\ldots,\overline{\alpha k},\overline{-\alpha k},\overline{\alpha k},\overline{b}) \\
                          &=& \begin{pmatrix}
   K_{n}(\overline{a},\overline{\alpha k},\overline{-\alpha k},\ldots,\overline{\alpha k},\overline{-\alpha k},\overline{\alpha k},\overline{b})   & -K_{n-1}(\overline{\alpha k},\overline{-\alpha k},\ldots,\overline{\alpha k},\overline{-\alpha k},\overline{\alpha k},\overline{b}) \\
   K_{n-1}(\overline{a},\overline{\alpha k},\overline{-\alpha k},\ldots,\overline{\alpha k},\overline{-\alpha k},\overline{\alpha k}) & -K_{n-2}(\overline{\alpha k},\overline{-\alpha k},\ldots,\overline{\alpha k},\overline{-\alpha k},\overline{\alpha k})
\end{pmatrix}.\\
\end{eqnarray*}
Donc, \[K_{n-1}(\overline{a},\overline{\alpha k},\overline{-\alpha k},\ldots,\overline{\alpha k},\overline{-\alpha k},\overline{\alpha k})=-K_{n-1}(\overline{\alpha k},\overline{-\alpha k},\ldots,\overline{\alpha k},\overline{-\alpha k},\overline{\alpha k},\overline{b})=\overline{0}\]et \[K_{n-2}(\overline{\alpha k},\overline{-\alpha k},\ldots,\overline{\alpha k},\overline{-\alpha k},\overline{\alpha k})=-\overline{\epsilon}.\] Or, 
\begin{eqnarray*}
K_{n-1}(\overline{a},\overline{\alpha k},\overline{-\alpha k},\ldots,\overline{\alpha k},\overline{-\alpha k},\overline{\alpha k}) 
 &=& \overline{a}K_{n-2}(\overline{\alpha k},\overline{-\alpha k},\ldots,\overline{\alpha k},\overline{-\alpha k},\overline{\alpha k}) \\
 &-& K_{n-3}(\overline{-\alpha k},\overline{\alpha k},\overline{-\alpha k},\ldots,\overline{\alpha k},\overline{-\alpha k},\overline{\alpha k}) \\
 &=& \overline{-\epsilon a}-K_{n-3}(\overline{-\alpha k},\overline{\alpha k},\overline{-\alpha k},\ldots,\overline{\alpha k},\overline{-\alpha k},\overline{\alpha k}). \\
\end{eqnarray*}
Ainsi, comme $\overline{\epsilon}^{2}=\overline{1}$, on a \[\overline{a}=\overline{-\epsilon}K_{n-3}(\overline{-\alpha k},\overline{\alpha k},\overline{-\alpha k},\ldots,\overline{\alpha k},\overline{-\alpha k},\overline{\alpha k}).\] De même, on a  
\begin{eqnarray*}
K_{n-1}(\overline{\alpha k},\overline{-\alpha k},\ldots,\overline{\alpha k},\overline{-\alpha k},\overline{\alpha k},\overline{b})
 &=& \overline{b}K_{n-2}(\overline{\alpha k},\overline{-\alpha k},\ldots,\overline{\alpha k},\overline{-\alpha k},\overline{\alpha k}) \\
 &-& K_{n-3}(\overline{\alpha k},\overline{-\alpha k},\ldots,\overline{\alpha k},\overline{-\alpha k}) \\
 &=& \overline{-\epsilon b}-K_{n-3}(\overline{\alpha k},\overline{-\alpha k},\ldots,\overline{\alpha k},\overline{-\alpha k}). \\
\end{eqnarray*}
Donc, 
\begin{eqnarray*}
\overline{b} &=& \overline{-\epsilon}K_{n-3}(\overline{\alpha k},\overline{-\alpha k},\ldots,\overline{\alpha k},\overline{-\alpha k}) \\
             &=& \overline{(-\epsilon)}\overline{(-1)^{n-3}}K_{n-3}(\overline{-\alpha k},\overline{\alpha k},\ldots,\overline{-\alpha k},\overline{\alpha k}) \\
						 &=& \overline{(-\epsilon)}K_{n-3}(\overline{-\alpha k},\overline{\alpha k},\ldots,\overline{-\alpha k},\overline{\alpha k})~{\rm car}~n-3~{\rm pair}\\
						 &=& \overline{a}. \\
\end{eqnarray*}
						
\noindent De plus, on a 
\begin{eqnarray*}
\overline{-\epsilon} &=& K_{n-2}(\overline{\alpha k},\overline{-\alpha k},\ldots,\overline{\alpha k},\overline{-\alpha k},\overline{\alpha k}) \\
                     &=& \overline{\alpha k}K_{n-3}(\overline{-\alpha k},\ldots,\overline{\alpha k},\overline{-\alpha k},\overline{\alpha k})-K_{n-4}(\overline{\alpha k},\overline{-\alpha k},\ldots,\overline{\alpha k},\overline{-\alpha k},\overline{\alpha k}) \\
\end{eqnarray*}
et 

$M_{n-2}(\overline{\alpha k},\overline{-\alpha k},\ldots,\overline{\alpha k},\overline{-\alpha k},\overline{\alpha k})$
\\
\\ $=\begin{pmatrix}
   K_{n-2}(\overline{\alpha k},\overline{-\alpha k},\ldots,\overline{\alpha k},\overline{-\alpha k},\overline{\alpha k})   & -K_{n-3}(\overline{-\alpha k},\overline{\alpha k},\ldots,\overline{-\alpha k},\overline{\alpha k}) \\
   K_{n-3}(\overline{\alpha k},\overline{-\alpha k},\ldots,\overline{\alpha k},\overline{-\alpha k}) & -K_{n-4}(\overline{-\alpha k},\ldots,\overline{\alpha k},\overline{-\alpha k})
\end{pmatrix}$
\\
\\ $=\begin{pmatrix}
   K_{n-2}(\overline{\alpha k},\overline{-\alpha k},\ldots,\overline{\alpha k},\overline{-\alpha k},\overline{\alpha k})   & -K_{n-3}(\overline{\alpha k},\overline{-\alpha k},\ldots,\overline{\alpha k},\overline{-\alpha k}) \\
   K_{n-3}(\overline{\alpha k},\overline{-\alpha k},\ldots,\overline{\alpha k},\overline{-\alpha k}) & -K_{n-4}(\overline{-\alpha k},\ldots,\overline{\alpha k},\overline{-\alpha k})
\end{pmatrix} \in SL_{2}(\mathbb{Z}/N\mathbb{Z})$.

\noindent Ainsi, \[-K_{n-2}(\overline{\alpha k},\overline{-\alpha k},\ldots,\overline{\alpha k},\overline{-\alpha k},\overline{\alpha k})K_{n-4}(\overline{-\alpha k},\ldots,\overline{\alpha k},\overline{-\alpha k})+K_{n-3}(\overline{\alpha k},\overline{-\alpha k},\ldots,\overline{\alpha k},\overline{-\alpha k})^{2}=\overline{1}.\] 
\\Or, comme $K_{n-2}(\overline{\alpha k},\overline{-\alpha k},\ldots,\overline{\alpha k},\overline{-\alpha k},\overline{\alpha k})=\overline{-\epsilon}$, on a \[\overline{\epsilon}K_{n-4}(\overline{-\alpha k},\ldots,\overline{\alpha k},\overline{-\alpha k})+K_{n-3}(\overline{\alpha k},\overline{-\alpha k},\ldots,\overline{\alpha k},\overline{-\alpha k})^{2}=\overline{1},\] c'est-à-dire \[K_{n-4}(\overline{-\alpha k},\ldots,\overline{\alpha k},\overline{-\alpha k})=\overline{\epsilon}(\overline{1}-K_{n-3}(\overline{\alpha k},\overline{-\alpha k},\ldots,\overline{\alpha k},\overline{-\alpha k})^{2}).\] 
\\Donc, on a
\begin{eqnarray*}
\overline{-\epsilon} &=& \overline{\alpha k}K_{n-3}(\overline{-\alpha k},\overline{\alpha k},\ldots,\overline{-\alpha k},\overline{\alpha k})+K_{n-4}(\overline{-\alpha k},\ldots,\overline{\alpha k},\overline{-\alpha k}) \\
                     &=& \overline{\alpha k}K_{n-3}(\overline{\alpha k},\overline{-\alpha k},\ldots,\overline{\alpha k},\overline{-\alpha k})+\overline{\epsilon}(\overline{1}-K_{n-3}(\overline{\alpha k},\overline{-\alpha k},\ldots,\overline{\alpha k},\overline{-\alpha k})^{2}) \\ 
										 &=& \overline{\alpha k}K_{n-3}(\overline{\alpha k},\overline{-\alpha k},\ldots,\overline{\alpha k},\overline{-\alpha k})+\overline{\epsilon}-\overline{\epsilon}K_{n-3}(\overline{\alpha k},\overline{-\alpha k},\ldots,\overline{\alpha k},\overline{-\alpha k})^{2} \\
										 &=&  \overline{-\alpha \epsilon k a}+\overline{\epsilon}-\overline{\epsilon a^{2}}.\\
\end{eqnarray*}

\noindent Donc, $\overline{-2\epsilon}=\overline{-\alpha \epsilon k a}-\overline{\epsilon a^{2}}=-\overline{\epsilon}\overline{a}(\overline{\alpha k}+\overline{a})$ et donc $\overline{2}=\overline{a}(\overline{\alpha k}+\overline{a})$.

\end{proof}

\begin{remark} 

{\rm Dans le cas où $n$ est pair, il est possible que $\overline{a} \neq \overline{0}$ et $\overline{a} \neq \pm \overline{k}$. Par exemple, si on pose $N=14$, le 18-uplet $(\overline{7},\overline{3},\overline{11},\overline{3},\overline{11},\overline{3},\overline{11},\overline{3},\overline{11},\overline{3},\overline{11},\overline{3},\overline{11},\overline{3},\overline{11},\overline{3},\overline{11},\overline{7})$ est solution de \eqref{p}.}

\end{remark}

\subsection{Preuve du théorème d'irréductibilité}

On peut maintenant démontrer le résultat principal.

\begin{proof}[Démonstration du théorème \ref{26}]

Soient $\overline{k} \in \mathbb{Z}/N\mathbb{Z}$, $\overline{k} \neq \overline{0}$, et $n \in \mathbb{N}^{*}$ tels que le $n$-uplet $(\overline{k},\overline{-k},\ldots,\overline{k},\overline{-k})$ d'éléments de $\mathbb{Z}/N\mathbb{Z}$ est une solution dynomiale minimale de \eqref{p}. On suppose que $\overline{k}^{2}+\overline{8}$ n'est pas un carré dans $\mathbb{Z}/N\mathbb{Z}$. On suppose par l'absurde que cette solution peut s'écrire comme une somme de deux solutions non triviales.
\\
\\$\overline{k} \neq \pm \overline{1}$ car $(\pm \overline{1})^{2}+\overline{8}=\overline{9}=\overline{3}^{2}$. Donc, si $n=4$, $(\overline{k},\overline{-k},\ldots,\overline{k},\overline{-k})$ est irréductible, puisque les solutions réductibles de \eqref{p} de taille 4 contiennent toujours $\pm \overline{1}$. On suppose maintenant $n \geq 6$.
\\
\\Il existe $(\overline{a_{1}},\ldots,\overline{a_{l}})$ et $(\overline{b_{1}},\ldots,\overline{b_{l'}})$ solutions de \eqref{p} différentes de $(\overline{0},\overline{0})$ avec $l+l'=n+2$ et $l,l' \geq 3$ tels que \[(\overline{k},\overline{-k},\ldots,\overline{k},\overline{-k}) \sim (\overline{b_{1}+a_{l}},\overline{b_{2}},\ldots,\overline{b_{l'-1}},\overline{b_{l'}+a_{1}},\overline{a_{2}},\ldots,\overline{a_{l-1}}).\] 
\\De plus, $\overline{k} \notin \{\overline{0},\overline{-1},\overline{1}\}$ donc $l,l'>3$. Il existe $\alpha$ dans $\{\pm 1\}$~ tel que \[(\overline{\alpha k},\overline{-\alpha k},\ldots,\overline{\alpha k},\overline{-\alpha k})=(\overline{b_{1}+a_{l}},\overline{b_{2}},\ldots,\overline{b_{l'-1}},\overline{b_{l'}+a_{1}},\overline{a_{2}},\ldots,\overline{a_{l-1}}).\] On a deux cas :

\begin{itemize}
\item \underline{Cas 1 :} Si $l$ est pair alors $l'=n+2-l$ est pair. On a donc \[(\overline{a_{1}},\ldots,\overline{a_{l}})=(\overline{a_{1}},\overline{\alpha k},\overline{-\alpha k},\ldots,\overline{\alpha k},\overline{-\alpha k},\overline{a_{l}}).\] Comme $(\overline{a_{1}},\ldots,\overline{a_{l}})$ est solution de \eqref{p}, on a, par le lemme \ref{42}, $\overline{a_{1}}=\overline{a}=\overline{-a_{l}}$ et $\overline{0}=\overline{a}(\overline{a}+\overline{\alpha k})$. Comme $N$ est premier, $\mathbb{Z}/N\mathbb{Z}$ est intègre et donc l'équation $\overline{0}=\overline{a}(\overline{a}+\overline{\alpha k})$ a pour solutions $\overline{a}=\overline{0}$ et $\overline{a}=\overline{ -\alpha k}$. 
\\
\\Si $\overline{a}=\overline{0}$ alors $(\overline{a_{2}},\ldots,\overline{a_{l-1}})=(\overline{\alpha k},\overline{-\alpha k},\ldots,\overline{\alpha k},\overline{-\alpha k}) \in (\mathbb{Z}/N\mathbb{Z})^{l-2}$ est encore solution de \eqref{p} ce qui contredit la minimalité de la solution (si $\alpha=-1$ alors $(\overline{-a_{2}},\ldots,\overline{-a_{l-1}})=(\overline{k},\overline{-k},\ldots,\overline{ k},\overline{-k})$ est encore solution). Donc, $\overline{a}=\overline{-\alpha k}$ et par minimalité de la solution on a $l \geq n$ ce qui implique $l' \leq 2$. Donc, $l'=2$ et $(\overline{b_{1}},\ldots,\overline{b_{l'}})=(\overline{0},\overline{0})$ ce qui est absurde.
\\
\item \underline{Cas 2 :} Si $l$ est impair alors $l'=n+2-l$ est impair. On a donc \[(\overline{a_{1}},\ldots,\overline{a_{l}})=(\overline{a_{1}},\overline{-\alpha k},\overline{\alpha k},\ldots,\overline{-\alpha k},\overline{\alpha k},\overline{-\alpha k},\overline{a_{l}})\] et \[(\overline{b_{1}},\ldots,\overline{b_{l'}})=(\overline{b_{1}},\overline{-\alpha k},\overline{\alpha k},\ldots,\overline{-\alpha k},\overline{\alpha k},\overline{-\alpha k},\overline{b_{l'}}).\]  Comme $(\overline{a_{1}},\ldots,\overline{a_{l}})$ et $(\overline{b_{1}},\ldots,\overline{b_{l'}})$ sont solutions de \eqref{p}, on a par le lemme \ref{42} :
\begin{itemize}[label=$\circ$]
\item $\overline{a_{1}}=\overline{a}=\overline{a_{l}}$ et $\overline{2}=\overline{a}(\overline{a}-\overline{\alpha k})$;
\item $\overline{b_{1}}=\overline{b}=\overline{b_{l'}}$ et $\overline{2}=\overline{b}(\overline{b}-\overline{\alpha k})$.
\\
\end{itemize}
$\overline{a}$ et $\overline{b}$ sont des racines de $P(X)=X(X-\overline{\alpha k})-\overline{2}=X^{2}-\overline{\alpha k}X-\overline{2}$. Comme $N$ est premier différent de 2, $\mathbb{Z}/N\mathbb{Z}$ est un corps de caractéristique différente de 2 et le discriminant de $P$ est $\Delta=(-\overline{\alpha k})^{2}-4\times \overline{-2}=\overline{k}^{2}+\overline{8}$. Comme $\overline{k}^{2}+\overline{8}$ n'est pas un carré dans $\mathbb{Z}/N\mathbb{Z}$, $P$ n'a pas de racine dans $\mathbb{Z}/N\mathbb{Z}$ ce qui est absurde.
\\
\end{itemize}

\noindent Ainsi, on arrive à une absurdité dans les deux cas et donc $(\overline{k},\overline{-k},\ldots,\overline{k},\overline{-k})$ est irréductible.

\end{proof}

\begin{remarks}

{\rm i) Il existe des solutions dynomiales minimales réductibles. Par exemple, posons $N=11$. La solution $\overline{2}$-dynomiale minimale (qui est de taille 12) est réductible. En effet, $(\overline{6},\overline{9},\overline{2},\overline{9},\overline{6})$ est solution de $(E_{11})$ et \[(\overline{2},\overline{9},\overline{2},\overline{9},\overline{2},\overline{9},\overline{2},\overline{9},\overline{2},\overline{9},\overline{2},\overline{9}) = (\overline{7},\overline{9},\overline{2},\overline{9},\overline{2},\overline{9},\overline{2},\overline{9},\overline{7}) \oplus (\overline{6},\overline{9},\overline{2},\overline{9},\overline{6}).\] Cet exemple montre que, contrairement aux solutions monomiales minimales, il existe des solutions dynomiales minimales réductibles même si $N$ est premier. De même, contrairement aux solutions monomiales minimales, il existe des solutions $\overline{2}$-dynomiales minimales réductibles.
\\
\\ii) La condition $\overline{k}^{2}+\overline{8}$ n'est pas un carré dans $\mathbb{Z}/N\mathbb{Z}$ n'est pas une condition nécessaire. Par exemple, $(\overline{6},\overline{-6},\overline{6},\overline{-6})$ est une solution dynomiale minimale irréductible pour $N=19$ et $\overline{6}^{2}+\overline{8}=\overline{36}+\overline{8}=\overline{6}=\overline{5}^{2}$.}

\end{remarks}

\subsection{Applications}
\label{app}

La condition $\overline{k}^{2}+\overline{8}$ n'est pas un carré dans $\mathbb{Z}/N\mathbb{Z}$ peut être vérifiée à l'aide la loi de réciprocité quadratique de Gauss (puisque $N$ est premier). Si $p$ est un nombre premier impair et si $a$ est un entier premier avec $p$, on note $\left(\dfrac{a}{p}\right)$ le symbole de Legendre c'est-à-dire :
\[\left(\dfrac{a}{p}\right)=\left\{
    \begin{array}{ll}
        1 & \mbox{si } \overline{a}~{\rm est~un~carr\acute{e}~dans}~\mathbb{Z}/p\mathbb{Z}; \\
        -1 & \mbox{sinon }.
    \end{array}
\right.  \\ \]

\noindent Le symbole de Legendre vérifie les propriétés suivantes :

\begin{lemma}[\cite{G}, proposition XII.20]
\label{43}
Soient $p$ un nombre premier impair et $a$ et $b$ deux entiers premiers avec $p$. On a :
\\i)~(critère d'Euler)~$\left(\dfrac{a}{p}\right) \equiv a^{\frac{p-1}{2}}~[p]$;
\\ii)~(multiplicativité)~$\left(\dfrac{ab}{p}\right)=\left(\dfrac{a}{p}\right)\left(\dfrac{b}{p}\right)$;
\\iii)~$\left(\dfrac{-1}{p}\right)=(-1)^{\frac{p-1}{2}}$.

\end{lemma}

\begin{theorem}[Loi de réciprocité quadratique de Gauss; \cite{G}, Théorème XII.25 ]
\label{44}

Soient $p$ et $q$ deux nombres premiers impairs distincts. On a \[\left(\dfrac{p}{q}\right)\left(\dfrac{q}{p}\right)= (-1)^{\frac{p-1}{2}\frac{q-1}{2}}.\]

\end{theorem}

Par exemple, si $\overline{k}=\overline{3}$, on a $\overline{k^{2}+8}=\overline{17}$. Comme 17 est premier, on peut utiliser la loi de réciprocité quadratique de Gauss pour savoir si $\overline{k^{2}+8}$ est un carré modulo $N$ (avec $N$ un entier premier supérieur à 5). On a alors par exemple :

\begin{proposition}
\label{45}

La solution $\overline{3}$-dynomiale minimale de $(E_{97})$ est irréductible.
	
\end{proposition}

\begin{proof}

97 est premier et $\overline{3} \notin \{\overline{0},\overline{1},\overline{-1}\}$. De plus, $97=5\times 17+12$ et les carrés modulo 17 sont :\[\{\overline{0},\overline{1},\overline{4},\overline{9},\overline{-1},\overline{8},\overline{2},\overline{15},\overline{13}\}.\] Donc, d'après la loi de réciprocité quadratique de Gauss, on a :
\[\left(\dfrac{17}{97}\right)=\left(\dfrac{97}{17}\right)(-1)^{\frac{17-1}{2}\frac{97-1}{2}}=\left(\dfrac{12}{17}\right)(-1)^{8 \times 48}=-1.\]
Donc, $\overline{3^{2}+8}$ n'est pas un carré modulo 97. Donc, par le théorème \ref{26}, la solution $\overline{3}$-dynomiale minimale de $(E_{97})$ est irréductible.

\end{proof}

On peut aussi s'intéresser aux cas de la solution $\overline{2}$-dynomiale minimale de \eqref{p} où $N$ est un entier premier supérieur à 5. On commence par le résultat intermédiaire ci-dessous :

\begin{lemma}
\label{46}

Soit $p$ un nombre premier. 
\[\overline{3}~{\rm est~un~carr\acute{e}~dans}~\mathbb{Z}/p\mathbb{Z} \Longleftrightarrow \left\{
    \begin{array}{ll}
        p=2;  \\
        p=3; \\
				p \equiv \pm 1~[12]. 
    \end{array}
\right.  \\ \]
	
\end{lemma}

\begin{proof}

Si $p=2$ alors $\overline{3}=\overline{1}=\overline{1}^{2}$ et si $p=3$ alors $\overline{3}=\overline{0}=\overline{0}^{2}.$ On suppose maintenant $p$ supérieur à 5. D'après la loi de réciprocité quadratique de Gauss on a \[\left(\dfrac{3}{p}\right)=\left(\dfrac{p}{3}\right)(-1)^{\frac{p-1}{2}\frac{3-1}{2}}=\left(\dfrac{p}{3}\right)(-1)^{\frac{p-1}{2}}.\]
Comme $p$ est premier, $p \equiv 1~[3]$ ou $p \equiv -1~[3]$. De plus, on a $\left(\dfrac{1}{3}\right)=1$ et $\left(\dfrac{-1}{3}\right)=-1$. On distingue les deux cas :
\begin{itemize}
\item Si $p=3k+1$ avec $k$ un entier naturel pair
\[\left(\dfrac{3}{p}\right)=1 \Longleftrightarrow \left(\dfrac{1}{3}\right)(-1)^{\frac{3k}{2}}=1 \Longleftrightarrow (-1)^{\frac{3k}{2}}=1 \Longleftrightarrow 4~{\rm divise}~k \Longleftrightarrow p \equiv 1~[12].\]
\item Si $p=3k-1$ avec $k$ un entier naturel pair 
\[\left(\dfrac{3}{p}\right)=1 \Longleftrightarrow \left(\dfrac{-1}{3}\right)(-1)^{\frac{3k-2}{2}}=1 \Longleftrightarrow (-1)^{\frac{3k}{2}}=1 \Longleftrightarrow 4~{\rm divise}~k \Longleftrightarrow p \equiv -1~[12].\]
\end{itemize}

\end{proof}

\begin{proposition}
\label{47}

Si $N$ est un entier premier supérieur à 5 tel que $N \not\equiv \pm 1[12]$ alors la solution $\overline{2}$-dynomiale minimale de \eqref{p} est irréductible.
	
\end{proposition}

\begin{proof}

Par le lemme précédent, $\overline{3}$ est un carré modulo $N$ si et seulement si $N \equiv \pm 1 [12]$.
\\
\\$N$ est premier. De plus, on a $\overline{2^{2}+8}=\overline{12}$ et donc par multiplicativité du symbole de Legendre :
\[\left(\dfrac{12}{N}\right)=\left(\dfrac{2^{2} \times 3}{N}\right)=\left(\dfrac{2^{2}}{N}\right)\left(\dfrac{3}{N}\right)=\left(\dfrac{2}{N}\right)^{2}\left(\dfrac{3}{N}\right)=\left(\dfrac{3}{N}\right)=-1.\]
Donc, par le théorème \ref{26}, la solution $\overline{2}$-dynomiale minimale de \eqref{p} est irréductible.

\end{proof}

\begin{remark}

{\rm Par le théorème faible de la progression arithmétique de Dirichlet (voir \cite{G} proposition VII.13), il existe une infinité de nombres premiers supérieurs à 5 congrus à 1 modulo 12.}

\end{remark}

La condition de la proposition précédente n'est pas nécessaire. Par exemple, si $N=59 \equiv -1~[12]$  alors la solution $\overline{2}$-dynomiale minimale de \eqref{p} est irréductible. En effet, celle-ci est de taille 20 et la seule façon de la réduire est de trouver une solution de \eqref{p} de la forme $(\overline{a},\overline{-2},\overline{2},\overline{-2},\ldots,\overline{2},\overline{-2},\overline{a})$ de taille inférieure à 19 (car les éléments donnés dans le cas 1 de la preuve du théorème \ref{26} sont toujours valides). Par le lemme \ref{42}, les seules valeurs de $\overline{a}$ possibles sont $\overline{12}$ et $\overline{-10}$. Or, en calculant toutes les possibilités on ne trouve aucune solution de \eqref{p}, ce qui implique l'irréductibilité de la solution $\overline{2}$-dynomiale minimale de $(E_{59})$. Ceci nous amène au problème ouvert suivant :

\begin{pro}

Soit $N$ un nombre premier. Trouver des conditions nécessaires et suffisantes sur $N$ pour l'irréductibilité de la solution $\overline{2}$-dynomiale minimale de \eqref{p}.

\end{pro}

On donne en annexe \ref{B} les éléments concernant la réductibilité ou l'irréductibilité des solutions $\overline{2}$-dynomiales minimales de \eqref{p} pour les nombres premiers congrus à $\pm 1$ modulo 12 et inférieurs à 500.

\medskip

\noindent {\bf Remerciements}.
Je remercie Valentin Ovsienko pour son aide précieuse.

\appendix

\section{Taille des solutions $\overline{k}$-monomiales minimales pour les nombres premiers entre 11 et 47}
\label{A}

\begin{center}
\begin{tabular}{|c|c|c|c|c|c|c|c|c|c|c|c|}
\hline
  \multicolumn{1}{|c|}{\backslashbox{$\overline{k}$}{\vrule width 0pt height 1.25em$N$}}     & 11 & 13 & 17 & 19 & 23  & 29 & 31 & 37 & 41 & 43 & 47  \rule[-7pt]{0pt}{18pt} \\
  \hline
  $\overline{0}$ & 2 & 2 & 2 & 2 & 2 & 2 & 2 & 2 & 2 & 2 & 2   \rule[-7pt]{0pt}{18pt} \\
	\hline
  $\overline{1}$ & 3 & 3 & 3 & 3 & 3 & 3 & 3 & 3 & 3 & 3 & 3   \rule[-7pt]{0pt}{18pt} \\
	\hline
  $\overline{2}$  & 11 & 13 & 17 & 19 & 23 & 29 & 31 & 37 & 41 & 43 & 47   \rule[-7pt]{0pt}{18pt} \\
	\hline
  $\overline{3}$  & 5 & 7 & 9 & 9 & 12 & 7  & 15 & 19 & 10 & 22 & 8   \rule[-7pt]{0pt}{18pt} \\
  \hline
	$\overline{4}$  & 5  & 6 & 9 & 5 & 11  & 15 & 16 & 18 & 7 & 11 & 23    \rule[-7pt]{0pt}{18pt} \\
	\hline
	 $\overline{5}$ & 6 & 7 & 8 & 5 &  4 & 5 & 8 & 9 & 20 & 21 & 23  \rule[-7pt]{0pt}{18pt} \\
	\hline
	 $\overline{6}$ & 6 & 7 & 4 & 10 & 11 & 5 & 15 & 19 & 5 & 22 & 23  \rule[-7pt]{0pt}{18pt} \\
	\hline
	 $\overline{7}$ & 5 & 7 & 9 & 9 & 6 & 7 & 15 & 19 & 5 & 11 & 4   \rule[-7pt]{0pt}{18pt} \\
	\hline
	$\overline{8}$  & 5 & 7 & 8 & 10 & 12 & 15 & 4 & 19 & 21 & 7 & 24   \rule[-7pt]{0pt}{18pt} \\
	\hline
	$\overline{9}$  & 11 & 6 & 8 & 9 & 11 & 15 & 16 & 9 & 20 & 11 & 24   \rule[-7pt]{0pt}{18pt} \\
	\hline
	$\overline{10}$ & 3 & 7 & 9 & 9 & 11 & 14 & 16 & 19 & 21 & 21 & 23  \rule[-7pt]{0pt}{18pt} \\
	\hline
	$\overline{11}$ &   & 13 & 4 & 10 & 11 & 7 &16 & 19 & 7 & 21 & 24   \rule[-7pt]{0pt}{18pt} \\
	\hline
	$\overline{12}$ &   & 3 & 8 & 9 & 11 & 14 & 5 & 19 & 21 & 21 & 6   \rule[-7pt]{0pt}{20pt} \\
	\hline
	$\overline{13}$ &   &  & 9 & 10 & 11 & 14 & 5 & 19 & 20 & 21 & 23   \rule[-7pt]{0pt}{18pt} \\
	\hline
	$\overline{14}$ &  &  & 9 & 5 & 11 & 15 & 8 & 9 & 7 & 11 & 23  \rule[-7pt]{0pt}{18pt} \\
	\hline
	$\overline{15}$ &  &  & 17 & 5 & 12 & 15 & 15 & 6 & 20 & 7 & 12   \rule[-7pt]{0pt}{18pt} \\
	\hline
	$\overline{16}$ &  &  & 3 & 9 & 6 & 14 & 15 & 18 & 21 & 22 & 23   \rule[-7pt]{0pt}{18pt} \\
	\hline
	$\overline{17}$ &  &  &  & 19 & 11 & 14 & 8 & 18 & 4 & 22 & 23   \rule[-7pt]{0pt}{18pt} \\
	\hline
	$\overline{18}$ &  &  &  & 3 & 4 & 7 & 5 & 19 & 10 & 22 & 8   \rule[-7pt]{0pt}{18pt} \\
	\hline
	$\overline{19}$ &  &  &  &  & 11 & 14 & 5 & 19 & 21 & 7 & 23  \rule[-7pt]{0pt}{18pt} \\
	\hline
	$\overline{20}$ &  &  &  &  & 12 & 15 & 16 & 18 & 21 & 21 & 24  \rule[-7pt]{0pt}{18pt} \\
	\hline
	$\overline{21}$ &  &  &  &  & 23 & 15 & 16 & 18 & 21 & 11 & 23   \rule[-7pt]{0pt}{18pt} \\
	\hline
	$\overline{22}$ &  &  &  &  & 3 & 7 & 16 & 6 & 21 & 11 & 12    \rule[-7pt]{0pt}{18pt} \\
	\hline
	$\overline{23}$ &  &  &  &  &  & 5 & 4 & 9 & 10 & 21 & 23  \rule[-7pt]{0pt}{18pt} \\
	\hline

\end{tabular}
\end{center}

\section{Éléments sur la réductibilité des solutions $\overline{2}$-dynomiales minimales de \eqref{p} pour les nombres premiers congrus à $\pm 1$ modulo 12 et inférieurs à 500}
\label{B}

\begin{center}
\begin{tabular}{|c|c|c|p{2.5cm}|c|}
\hline
                  N & taille & réductibilité & \centering solutions de $X(X-\overline{2})=\overline{2}$ & exemple de solution permettant de réduire  \rule[-7pt]{0pt}{18pt} \\
  \hline
  11 & 12 & réductible & \centering $\overline{6},\overline{7}$ & $(\overline{6},\overline{-2},\overline{2},\overline{-2},\overline{6})$     \rule[-7pt]{0pt}{18pt} \\
	\hline
  13 & 14 & réductible & \centering $\overline{5},\overline{10}$  & $(\overline{5},\overline{-2},\overline{2},\overline{-2},\overline{5})$     \rule[-7pt]{0pt}{18pt} \\
	\hline
  23  & 22 & réductible & \centering $\overline{8},\overline{17}$ & $(\overline{17},\overline{-2},\overline{2},\overline{-2},\overline{2},\overline{-2},\overline{17})$    \rule[-7pt]{0pt}{18pt} \\
	\hline
  37  & 38 & réductible & \centering $\overline{16},\overline{23}$ & $(\overline{16},\overline{-2},\overline{2},\overline{-2},\overline{2},\overline{-2},\overline{2},\overline{-2},\overline{16})$     \rule[-7pt]{0pt}{18pt} \\
  \hline
	47  & 46  & réductible & \centering $\overline{13},\overline{36}$ & $(\overline{36},\overline{-2},\overline{2},\overline{-2},\ldots,\overline{2},\overline{-2},\overline{36}) \in (\mathbb{Z}/47\mathbb{Z})^{21}$      \rule[-7pt]{0pt}{18pt} \\
	\hline
	59 & 20 & irréductible & \centering $\overline{12},\overline{49}$ & \diagbox[width=80mm,height=7mm]{}{}  \rule[-7pt]{0pt}{18pt} \\
	\hline
	61 & 62 & réductible & \centering $\overline{9},\overline{54}$ & $(\overline{9},\overline{-2},\overline{2},\overline{-2},\overline{2},\overline{-2},\overline{2},\overline{-2},\overline{2},\overline{-2},\overline{9})$  \rule[-7pt]{0pt}{18pt} \\
	\hline
	71 & 70 & réductible & \centering $\overline{29},\overline{44}$ &  $(\overline{29},\overline{-2},\overline{2},\overline{-2},\overline{2},\overline{-2},\overline{29})$  \rule[-7pt]{0pt}{18pt} \\
	\hline
	73  & 36 & irréductible & \centering $\overline{22},\overline{53}$ & \diagbox[width=80mm,height=7mm]{}{}  \rule[-7pt]{0pt}{18pt} \\
	\hline
	83 & 84 & réductible & \centering $\overline{-12},\overline{14}$ &  $(\overline{14},\overline{-2},\overline{2},\overline{-2},\overline{2},\overline{-2},\overline{2},\overline{-2},\overline{2},\overline{-2},\overline{2},\overline{-2},\overline{14})$   \rule[-7pt]{0pt}{18pt} \\
	\hline
	97 & 48 & réductible & \centering $\overline{-9},\overline{11}$ &  $(\overline{11},\overline{-2},\overline{2},\overline{-2},\ldots,\overline{2},\overline{-2},\overline{11}) \in (\mathbb{Z}/97\mathbb{Z})^{29}$   \rule[-7pt]{0pt}{18pt} \\
	\hline
	 107 & 108 & réductible & \centering $\overline{-17},\overline{19}$ & $(\overline{19},\overline{-2},\overline{2},\overline{-2},\ldots,\overline{2},\overline{-2},\overline{19}) \in (\mathbb{Z}/107\mathbb{Z})^{23}$    \rule[-7pt]{0pt}{18pt} \\
	\hline
	109 & 110 & réductible & \centering $\overline{-59},\overline{61}$ & $(\overline{61},\overline{-2},\overline{2},\overline{-2},\ldots,\overline{2},\overline{-2},\overline{61}) \in (\mathbb{Z}/109\mathbb{Z})^{23}$     \rule[-7pt]{0pt}{20pt} \\
	\hline
	 131 & 132 & réductible & \centering $\overline{-92},\overline{94}$ & $(\overline{94},\overline{-2},\overline{2},\overline{-2},\ldots,\overline{2},\overline{-2},\overline{94}) \in (\mathbb{Z}/131\mathbb{Z})^{31}$    \rule[-7pt]{0pt}{18pt} \\
	\hline
	157 & 158 & réductible & \centering $\overline{-84},\overline{86}$ & $(\overline{-84},\overline{-2},\overline{2},\overline{-2},\ldots,\overline{2},\overline{-2},\overline{-84}) \in (\mathbb{Z}/157\mathbb{Z})^{59}$    \rule[-7pt]{0pt}{18pt} \\
	\hline
	167 & 166 & réductible & \centering $\overline{-104},\overline{106}$ & $(\overline{-104},\overline{-2},\overline{2},\overline{-2},\ldots,\overline{2},\overline{-2},\overline{-104}) \in (\mathbb{Z}/167\mathbb{Z})^{13}$     \rule[-7pt]{0pt}{18pt} \\
	\hline
	 179 & 36 & irréductible & \centering $\overline{-18},\overline{20}$ & \diagbox[width=80mm,height=7mm]{}{}    \rule[-7pt]{0pt}{18pt} \\
	\hline
	 181 & 182 & réductible & \centering $\overline{-32},\overline{34}$ & $(\overline{34},\overline{-2},\overline{2},\overline{-2},\ldots,\overline{2},\overline{-2},\overline{34}) \in (\mathbb{Z}/181\mathbb{Z})^{21}$     \rule[-7pt]{0pt}{18pt} \\
	\hline
	191 & 190 & réductible & \centering $\overline{-23},\overline{25}$ &  $(\overline{25},\overline{-2},\overline{2},\overline{-2},\ldots,\overline{2},\overline{-2},\overline{25}) \in (\mathbb{Z}/191\mathbb{Z})^{69}$    \rule[-7pt]{0pt}{18pt} \\
	\hline
	193 & 96 & irréductible & \centering $\overline{-13},\overline{15}$ & \diagbox[width=80mm,height=7mm]{}{}    \rule[-7pt]{0pt}{18pt} \\
	\hline
	227 & 76 & irréductible & \centering $\overline{-49},\overline{51}$ & \diagbox[width=80mm,height=7mm]{}{}   \rule[-7pt]{0pt}{18pt} \\
	\hline
	229 & 46 & irréductible & \centering $\overline{-157},\overline{159}$ & \diagbox[width=80mm,height=7mm]{}{}    \rule[-7pt]{0pt}{18pt} \\
	\hline
	239 & 14 & irréductible & \centering $\overline{-132},\overline{134}$ & \diagbox[width=80mm,height=7mm]{}{}    \rule[-7pt]{0pt}{18pt} \\
	\hline
	241 & 40 & irréductible & \centering $\overline{-55},\overline{57}$ & \diagbox[width=80mm,height=7mm]{}{}  \rule[-7pt]{0pt}{18pt} \\
  \hline
  251 & 84 & irréductible & \centering $\overline{-174},\overline{176}$ & \diagbox[width=80mm,height=7mm]{}{}  \rule[-7pt]{0pt}{18pt} \\
	\hline
  263 & 262 & réductible & \centering $\overline{-22},\overline{24}$ & $(\overline{24},\overline{-2},\overline{2},\overline{-2},\ldots,\overline{2},\overline{-2},\overline{24}) \in (\mathbb{Z}/263\mathbb{Z})^{85}$      \rule[-7pt]{0pt}{18pt} \\
	\hline
   277 & 278 & réductible & \centering $\overline{-146},\overline{148}$ &  $(\overline{148},\overline{-2},\overline{2},\overline{-2},\ldots,\overline{2},\overline{-2},\overline{148}) \in (\mathbb{Z}/277\mathbb{Z})^{83}$     \rule[-7pt]{0pt}{18pt} \\
	\hline
   311 & 310 & réductible & \centering $\overline{-24},\overline{26}$  & $(\overline{26},\overline{-2},\overline{2},\overline{-2},\ldots,\overline{2},\overline{-2},\overline{26}) \in (\mathbb{Z}/311\mathbb{Z})^{31}$    \rule[-7pt]{0pt}{18pt} \\
  \hline
	 313 & 78  & irréductible & \centering $\overline{-229},\overline{231}$ & \diagbox[width=80mm,height=7mm]{}{}       \rule[-7pt]{0pt}{18pt} \\
	\hline
	 337 & 28 & irréductible & \centering $\overline{-44},\overline{46}$ & \diagbox[width=80mm,height=7mm]{}{}  \rule[-7pt]{0pt}{18pt} \\
	\hline

\end{tabular}
\end{center}

\begin{center}
\begin{tabular}{|c|c|c|p{2.5cm}|c|}
\hline
                  N & taille & réductibilité & \centering solutions de $X(X-\overline{2})=\overline{2}$ & exemple de solution permettant de réduire  \rule[-7pt]{0pt}{18pt} \\
	\hline
	347 & 348 & réductible & \centering $\overline{-251},\overline{253}$ &  $(\overline{-251},\overline{-2},\overline{2},\overline{-2},\ldots,\overline{2},\overline{-2},\overline{-251}) \in (\mathbb{Z}/347\mathbb{Z})^{165}$  \rule[-7pt]{0pt}{18pt} \\
	\hline
	 349 & 350 & réductible & \centering $\overline{-183},\overline{185}$ &  $(\overline{-183},\overline{-2},\overline{2},\overline{-2},\ldots,\overline{2},\overline{-2},\overline{-183}) \in (\mathbb{Z}/349\mathbb{Z})^{147}$   \rule[-7pt]{0pt}{18pt} \\
	\hline
	 359 & 358 & réductible & \centering $\overline{-195},\overline{197}$ & $(\overline{-195},\overline{-2},\overline{2},\overline{-2},\ldots,\overline{2},\overline{-2},\overline{-195}) \in (\mathbb{Z}/359\mathbb{Z})^{103}$    \rule[-7pt]{0pt}{18pt} \\
	\hline
	 373 & 374 & réductible & \centering $\overline{-241},\overline{243}$ & $(\overline{-241},\overline{-2},\overline{2},\overline{-2},\ldots,\overline{2},\overline{-2},\overline{-241}) \in (\mathbb{Z}/373\mathbb{Z})^{111}$    \rule[-7pt]{0pt}{18pt} \\
	\hline
	383 & 382 & réductible & \centering $\overline{-223},\overline{225}$ & $(\overline{225},\overline{-2},\overline{2},\overline{-2},\ldots,\overline{2},\overline{-2},\overline{225}) \in (\mathbb{Z}/383\mathbb{Z})^{111}$  \rule[-7pt]{0pt}{18pt} \\
	\hline
	397 & 398 & réductible & \centering $\overline{-19},\overline{21}$ &  $(\overline{21},\overline{-2},\overline{2},\overline{-2},\ldots,\overline{2},\overline{-2},\overline{21}) \in (\mathbb{Z}/397\mathbb{Z})^{131}$   \rule[-7pt]{0pt}{18pt} \\
	\hline
	409 & 102  & réductible & \centering $\overline{-240},\overline{242}$ & $(\overline{242},\overline{-2},\overline{2},\overline{-2},\ldots,\overline{2},\overline{-2},\overline{242}) \in (\mathbb{Z}/409\mathbb{Z})^{95}$    \rule[-7pt]{0pt}{20pt} \\
	\hline
	419 & 140 & réductible & \centering $\overline{-28},\overline{30}$ &  $(\overline{30},\overline{-2},\overline{2},\overline{-2},\ldots,\overline{2},\overline{-2},\overline{30}) \in (\mathbb{Z}/419\mathbb{Z})^{45}$   \rule[-7pt]{0pt}{18pt} \\
	\hline
	421 & 422 & réductible & \centering $\overline{-73},\overline{75}$ & $(\overline{75},\overline{-2},\overline{2},\overline{-2},\ldots,\overline{2},\overline{-2},\overline{75}) \in (\mathbb{Z}/421\mathbb{Z})^{207}$  \rule[-7pt]{0pt}{18pt} \\
	\hline
	431 & 430 & réductible & \centering $\overline{-35},\overline{37}$ &  $(\overline{37},\overline{-2},\overline{2},\overline{-2},\ldots,\overline{2},\overline{-2},\overline{37}) \in (\mathbb{Z}/431\mathbb{Z})^{25}$  \rule[-7pt]{0pt}{18pt} \\
	\hline
	433 & 216 & réductible & \centering $\overline{-50},\overline{52}$ &  $(\overline{52},\overline{-2},\overline{2},\overline{-2},\ldots,\overline{2},\overline{-2},\overline{52}) \in (\mathbb{Z}/433\mathbb{Z})^{123}$  \rule[-7pt]{0pt}{18pt} \\
	\hline
	443 & 148 & irréductible & \centering $\overline{-271},\overline{273}$ &  \diagbox[width=80mm,height=7mm]{}{}   \rule[-7pt]{0pt}{18pt} \\
	\hline
	457 & 114 & irréductible & \centering $\overline{-311},\overline{313}$ & \diagbox[width=80mm,height=7mm]{}{}   \rule[-7pt]{0pt}{18pt} \\
	\hline
	467 & 468 & réductible & \centering $\overline{-300},\overline{302}$ &  $(\overline{302},\overline{-2},\overline{2},\overline{-2},\ldots,\overline{2},\overline{-2},\overline{302}) \in (\mathbb{Z}/467\mathbb{Z})^{65}$  \rule[-7pt]{0pt}{18pt} \\
	\hline
	479 & 478 & réductible & \centering $\overline{-30},\overline{32}$ &  $(\overline{32},\overline{-2},\overline{2},\overline{-2},\ldots,\overline{2},\overline{-2},\overline{32}) \in (\mathbb{Z}/479\mathbb{Z})^{269}$  \rule[-7pt]{0pt}{18pt} \\
	\hline
	491 & 492 & réductible & \centering $\overline{-112},\overline{114}$ &  $(\overline{114},\overline{-2},\overline{2},\overline{-2},\ldots,\overline{2},\overline{-2},\overline{114}) \in (\mathbb{Z}/491\mathbb{Z})^{323}$   \rule[-7pt]{0pt}{18pt} \\
	\hline

\end{tabular}
\end{center}

\end{document}